\documentclass[a4paper,oneside,12pt]{scrartcl}
\usepackage[english]{babel}
\usepackage{graphicx}
\usepackage{dsfont}
\usepackage{bm}
\usepackage{amssymb}
\usepackage{amsthm}
\usepackage{amsmath}
\usepackage{booktabs}
\usepackage{amsbsy}
\usepackage{fontenc}
\usepackage[utf8]{inputenc}
\usepackage{fancybox}
\usepackage[hypcap=false]{caption}
\usepackage{float}
\usepackage{subfigure} 
\usepackage{lmodern}
\usepackage[numbib,nottoc]{tocbibind}
\usepackage[pdfstartview={FitH},linkbordercolor={0 1 0}]{hyperref}
\usepackage{esint}
\usepackage{fancyhdr}
\usepackage{pdfpages}
\usepackage{trfsigns}
\usepackage{wasysym} 
\usepackage{siunitx} 
\usepackage{eqnarray} 
\usepackage{mathtools} 
\usepackage{relsize} 
\usepackage{scrhack} 

\allowdisplaybreaks 

\def\be{\begin{equation}}
\def\ee{\end{equation}}
\def\bs{\begin{subequations}}
\def\es{\end{subequations}}
\def\ba#1\ea{\begin{align}#1\end{align}}
\def\bes{\begin{equation*}}
\def\ees{\end{equation*}}
\def\bas#1\eas{\begin{align*}#1\end{align*}}

\theoremstyle{plain}
\newtheorem{theorem}{Theorem}[section]
\newtheorem{corollary}[theorem]{Corollary}
\newtheorem{lemma}[theorem]{Lemma}
\newtheorem{proposition}[theorem]{Proposition}
\newtheorem{conjecture}[theorem]{Conjecture}

\theoremstyle{remark}

\newtheorem*{motivation}{MOTIVATION}

\theoremstyle{definition}
\newtheorem{definition}[theorem]{Definition}

\newtheorem{remark}[theorem]{Remark}

\newcommand{\fullname}{Simon-Raphael Fischer}

\DeclareFontFamily{U}{mathx}{\hyphenchar\font45}
\DeclareFontShape{U}{mathx}{m}{n}{
      <5> <6> <7> <8> <9> <10>
      <10.95> <12> <14.4> <17.28> <20.74> <24.88>
      mathx10
      }{}
\DeclareSymbolFont{mathx}{U}{mathx}{m}{n}
\DeclareFontSubstitution{U}{mathx}{m}{n}
\DeclareMathAccent{\widecheck}{0}{mathx}{"71}
\DeclareMathAccent{\wideparen}{0}{mathx}{"75}

\begin{document}
\pagenumbering{Roman}
\renewcommand{\thefootnote}{\fnsymbol{footnote}}
\begin{titlepage}

\author{Simon-Raphael Fischer\footnote{email: Simon-Raphael.Fischer@unige.ch}} 
\title{Existence of Thin Shell Wormholes using non-linear distributional geometry} 
\date{\today} 
\maketitle
\thispagestyle{empty}

\begin{center}
Master thesis written at: \\
Mathematical Institute, Ludwig Maximilians University Munich \\
Theresienstr. 39, 80333 Munich, Germany \\
\ \\
Under the supervision of Prof. Dr. Mark John David Hamilton \\
Institute for Geometry and Topology, University of Stuttgart \\
Pfaffenwaldring 57, 70569 Stuttgart, Germany \\
\ \\
\ \\
\ \\
\textbf{Abstract}
\begin{abstract}
  \small{We study for which polynomials $F$ a thin shell wormhole with a continuous metric (connecting two Schwarzschild spacetimes of the same mass) satisfy the null energy condition (NEC) in $F(R)$-gravity. We avoid junction conditions by using the mathematical framework of the Colombeau algebra which describes a generalized framework of the distributional geometry such that one can define multiplications between distributions generalizing the tensor product of smooth tensors. The aim for physics is to motivate a conjecture about the satisfaction of the NEC for suitable quadratic $F$ while the aim for mathematics is to derive a rigorous framework describing this situation. Here the $F(R)$-gravity should be seen as a toy model, important is that the NEC may be satisfied by some form of "microstructure" which does not arise in the classical setting and may have interesting physical meanings.}
 \end{abstract}
\end{center}

\end{titlepage}


\tableofcontents

\renewcommand{\thefootnote}{\arabic{footnote}}

\setlength{\parindent}{12 pt}


\pagenumbering{arabic}

\section{Introduction}

The research subject of this paper is the study of the physical existence of thin-shell
wormholes in generalized theories of gravity by using the so-called "microstructure". Wormholes can be viewed as short cuts in a
spacetime or as a connection of two different spacetimes and are not only subject in science
fiction literature, there are also many scientific papers about them like \cite{energiedichte} and the book
about Lorentzian Wormholes of Matt Visser, see \cite{Visser}. In \cite{Visser} it is shown that wormholes
consists of exotic matter, i.e. matter which does not satisfy standard energy conditions
like the null energy condition. But in papers like \cite{Dilaton}, \cite{Referee2} and \cite{DGP} it was shown that
this can change in generalized theories of gravity such that the question arises in which
theories about gravity a fixed energy condition might be satisfied by wormholes. In this paper we will state some similar conjecture about the satisfaction of some fixed energy condition (here: the null energy condition) in some form of generalized gravity (here: $F(R)$-gravity for suitable quadratic $F$).

We will look at thin shell wormholes, i.e. wormholes produced by matter sitting on
a submanifold of codimension one, the thin shell like in \cite{Visser} (Chapter 14 ff.). In our case the matter will sit on
a 2-sphere for a fixed moment of time, defined as the boundary of a three dimensional ball, and this matter will connect two Schwarzschild
spacetimes of the same mass. The generalized theory of gravity will be in our case the
so-called $F(R)$-gravity where $F$ is some suitable polynomial defined on $R$ (the scalar
curvature). $F(R)$-gravity has its origin in cosmology when one tries to get a description
of the universe without dark matter, see \textit{e.g.} the introduction of \cite{f(R)} and the references therein. Some physical basics (mainly the mathematical formulation of the basic definitions) will be shortly introduced
in the second section. But one has to beware if $F(R)$-gravity can be well defined; in
our case we will have a metric which is only continuous at the thin shell such that the
scalar curvature will be a delta distribution with support on the thin shell. Since there is
no generic definition of multiplications in the distributional framework (especially not of
powers of the delta distribution) one has seemingly a problem to combine $F(R)$-gravity
with our taken thin shell wormhole. In \cite{junction} junction conditions were derived which
say under which conditions one can well-define $F(R)$-gravity, basically these conditions
consist of regularity conditions about the metric under which the problematic multiplications do not arise. The result
there is also that a continuous metric has too less regularity for $F$ with at least quadratic powers.

But we will present a way to avoid the junction conditions by generalizing the distributional
framework. This will be done in the third section after a short discussion of
the definition about the composition of the delta distribution with a suitable function.
There we will shortly introduce the so-called Colombeau algebra which generalizes the
distributional framework in such a way that one has a definition of a multiplication between
any distribution such that the tensor product of smooth tensors is generalized and
we will also mention why one can only generalize the tensor product of smooth tensors following mainly \cite{VickerAnfang}, \cite{Sobolev}, \cite{GrundlagenI} and \cite{GrundlagenII}.
This algebra basically replaces every distribution by a sequence of smooth tensors approximating
the initial distribution and then one defines everything naturally on these
sequences, especially one can define the multiplication as a pointwise multiplication of these sequences. But one has to define a quotient structure on this sequence space to get
a generalized tensor product of smooth tensors because one has two natural approximations for smooth tensors (one given by a convolution with a strict delta net and
one given by the constant sequence since smooth tensors do not have to be smoothed) which are in general different to each other. The quotient structure is then naturally defined by identifying both approaches (for which it can be shown that this is well-defined; see the references above).
Afterwards we will present the principle of association denoted by $\approx$ which may be used
to simplify arising terms and to get back to distributions by basically taking the limit of
the sequences (which may not exist).

Since we want to study some energy condition we have to make sense of non-negativity
in the Colombeau algebra. Therefore we will generalize non-negativity in the fourth
section. We will do this by using measure theory and the well-known representation
theorem of Riesz-Markov-Saks-Kakutani which basically identifies signed measures with
distributions of zero order.\footnote{Roughly spoken distributions of zero order are distributions whose domain of definition can be extended
to continuous functions with compact support, like the Dirac Delta distribution; normally every distribution is in general only defined on smooth
functions with compact support.} Basically a generalized sense of non-negativity is given by a
vanishing negative part, i.e. the negative part of the studied sequence shall approximate
the zero distribution. This has the advantage that one can extend the set of non-negative
elements by high singular terms which can not be described as non-negative in sense of
the classical sense of non-negativity due to their singularity.

The studied energy condition will be the so-called null energy condition (NEC) which
basically says that the mass-energy density measured by lightlike observers shall be
non-negative. To define such an energy condition in the generalized framework of the
Colombeau algebra one also has to make sense of lightlike vector fields which will be also
done in the fourth section. Finally we can then define the generalized NEC and the last
part of the fourth section shows a simplification of the arising NEC for some type of
stress-energy tensors (Hawking-Ellis type I, see \cite{Hawking}).

The fifth section then combines both, the physical and the mathematical basics. I.e.
we will study the NEC in our taken physical situation of a thin shell wormhole in $F(R)$-
gravity where we will take quadratic $F$. By using an ansatz of minimising the so-called
microstructure we are going to end this paper by stating the conjecture that the NEC is
satisfied in this setting. The microstructure basically describes the new (possible physical)
information given in the generalized setting which is in our case an infinitesimal internal
structure along the proper radial distance of the wormhole. We will see that one gets this microstructure due to singularities in the underlying theory; as a short glimpse roughly like that: Due to these singularities one has in general problems to well-define multiplications between distributions, \textit{e.g.} squares of the scalar curvature described by a delta distribution with support on the thin shell wormhole (extended along the time axis). In our case the thin shell will be a sphere with a radius big enough to prevent event horizons. Thus, as mentioned, we approximate this delta distribution with a smooth sequence which roughly represents a sequence of smooth wormholes supported on a spherical shell approximating our wormhole supported on the sphere. Since we have then a spherical shell instead of a sphere, this matter sequence will not just be described simply by a delta distribution, the sequence allows more complicated matter, especially matter with an internal structure varying along the radial structure which one does not have in the thin shell case where only constant and trivial values are possible.\footnote{Due to the nature of the delta distribution one looses any information off the support of the delta distribution by $f(x) ~ \delta(x) = f(0) ~ \delta(x)$, \textit{i.e.} structure of the matter along $x$ is lost and trivial while one can keep such information in the Colombeau algebra in form of microstructure.} In the standard case (\textit{i.e.} taking the limit of the sequences, which is possible then) this structure vanishes but in our case with the arising singularities this structure remains (and may be infinity in the limit; see later) because the sequence might be multiplied with such a singularity such that the product may not vanish in the limit. This is named as microstructure because it comes from the infinitesimal information carried along the sequence itself and not from their limit since it vanishes in the limit without such multiplications. Thence one can hope to simulate more complicated internal structures for matters in thin shells, strings and so on.

We will also mention there which
problems one has to face when one tries to prove the conjecture rigorously.

\section{Physical situation}

We will study a thin shell wormhole connecting two static spherical symmetric spacetimes where we will follow mainly \cite{Visser} (chapter 15, page 176 ff.). So we study two spacetimes $M^\pm := \mathds{R} \times [a,\infty) \times \mathds{S}^2 \ni (t_\pm, r_\pm, \vartheta_\pm, \varphi_\pm)$ (the typical Schwarzschild coordinates with $\mathds{S}^2 \ni (\vartheta_\pm, \varphi_\pm)$ the 2-sphere), $a> 0$ big enough to prevent event horizons, with metrics
\ba
g^\pm := - A^\pm(r_\pm) ~ \mathrm{d}t_\pm^2 + B^\pm(r_\pm) ~ \mathrm{d}r_\pm^2 + r_\pm^2 ~ \mathrm{d}\vartheta_\pm^2 + r_\pm^2 \sin^2(\vartheta_\pm) ~ \mathrm{d}\varphi_\pm^2,
\ea
where $A^\pm, B^\pm \in C^\infty([a,\infty); (0,\infty))$\footnote{\textit{I.e.} positive smooth functions defined on $[a, \infty)$.} such that they converge to some constant at infinity, \textit{e.g.} converging to 1 to get an asymptotically flat spacetime. One sees that both spacetimes are coming from the manifold $\mathds{R} \times \left(\mathds{R} \setminus \{0\}\right) \times \mathds{S}^2$, \textit{e.g.} equipped with the Schwarzschild spacetime, in which one has cut out a ball with radius $a$ for every constant time slice. Both spacetimes can now be glued together along their boundary $\Omega^\pm := \left\{ (t_\pm, r_\pm, \vartheta_\pm, \varphi_\pm) \in M^\pm ~ \middle| ~ r_\pm=a\right\}$ which will be done in Def. \ref{def:thinshellwormhole}; in essence one is doing a connected sum and we are now doing a connected sum of the initial metric such that the resulting overall metric is at least continuous. We will not use the definition of a connected sum directly; we will use coordinates instead because one has to calculate curvature terms at the end which has to be done in coordinates in this situation. For this we need to replace the radial coordinates $r_\pm$ by a proper radial distance $\tilde{\eta}$ which describes the glued radial distances $r_\pm$ such that one gets one global radial coordinate allowing negative values. For this let us define $\tilde{\eta}_+: [a, \infty) \to [0, \infty)$ and $\tilde{\eta}_-: [a, \infty) \to (-\infty,0]$ given by
\ba
\tilde{\eta}_\pm(r_\pm) &:= \pm \int_a^{r_\pm} \sqrt{B^\pm(r)} ~ \mathrm{d}r,
\label{eq:rinverse}
\ea
where we write $r_\pm$ for the argument because of bookkeeping reasons. Therefore one could view $\tilde{\eta}$ as a function on $M^+ \sqcup M^-$ given by $\tilde{\eta}_+$ on $M^+$ and $\tilde{\eta}_-$ on $M^-$, basically one only has to know that one gets a new radial coordinate $\tilde{\eta} \in \mathds{R}$, now also allowing negative values. Also observe that Eq. \eqref{eq:rinverse} can clearly be inverted because we have (a one-dimensional case and)
\ba
\frac{\mathrm{d}\tilde{\eta}_\pm}{\mathrm{d}r_\pm} &= \pm \sqrt{B^\pm(r_\pm)} \neq 0,
\label{eq:jump}
\ea
we also have then
\ba 
\mathrm{d}\tilde{\eta}_\pm^2 &= B^\pm(r_\pm) ~ \mathrm{d}r_\pm^2.
\ea
By inserting this into the metric we can now define the overall manifold defining the wormhole which we want to study.

\begin{definition}[Thin Shell wormhole]
\leavevmode\newline
Let $M := \mathds{R} \times \mathds{R} \times \mathds{S}^2$ with the standard differentiable structure (and the orientation is arbitrary, does not have to be the standard orientation if you do not want to) where the atlas contains at least the charts $\psi_\pm: M \to \mathds{R} \times \mathds{R} \times \mathds{S}^2$ defined by $\psi_- := \mathrm{Id}_M$ with values (the coordinates) $(t_-, \tilde{\eta}, u)$ and $\psi_+ := \left( \sqrt{\frac{A^-(a)}{A^+(a)}} ~ \mathrm{Id}_\mathds{R}\right) \oplus \mathrm{Id}_\mathds{R} \oplus \mathrm{Id}_{\mathds{S}^2}$ with values $(t_+, \tilde{\eta}, u)$ (which is clearly in the differentiable structure induced by $\psi_-$) and for $u \in \mathds{S}^2$ we write locally $\vartheta, \varphi$ for the angles as previously. Thus, we have a matching condition for the times measured on both sides of the universe
\begin{subequations}
\begin{align}
\psi_- \circ \psi_+^{-1} &= \left( \sqrt{\frac{A^+(a)}{A^-(a)}} ~ \mathrm{Id}_\mathds{R}\right) \oplus \mathrm{Id}_\mathds{R} \oplus \mathrm{Id}_{\mathds{S}^2}, \\
\text{meaning } t_- &= \sqrt{\frac{A^+(a)}{A^-(a)}} ~ t_+. 
\end{align}
\end{subequations}
To make $M$ to a spacetime we equip it with the following metric like motivated before 
\bs
\begin{align}
g^+ &= - A^+(r_+(\tilde{\eta})) ~ \mathrm{d}t_+^2 + \mathrm{d}\tilde{\eta}^2 + r_+^2(\tilde{\eta}) ~ \mathrm{d}\vartheta^2 + r_+^2(\tilde{\eta}) \sin^2(\vartheta) ~ \mathrm{d}\varphi^2 & \text{on } M^+, \\
g^- &= - A^-(r_-(\tilde{\eta})) ~ \mathrm{d}t_-^2 + \mathrm{d}\tilde{\eta}^2 + r_-^2(\tilde{\eta}) ~ \mathrm{d}\vartheta^2 + r_-^2(\tilde{\eta}) \sin^2(\vartheta) ~ \mathrm{d}\varphi^2 \nonumber \\
&= - A^-(r_-(\tilde{\eta})) ~ \frac{A^+(a)}{A^-(a)} ~ \mathrm{d}t_+^2 + \mathrm{d}\tilde{\eta}^2 + r_-^2(\tilde{\eta}) ~ \mathrm{d}\vartheta^2 + r_-^2(\tilde{\eta}) \sin^2(\vartheta) ~ \mathrm{d}\varphi^2 & \text{on } M^- ,
\end{align}
\es
where\footnote{Diffeomorphic to the $M^\pm$ of the previous discussion but with changed coordinates like illustrated before.} $M^+  := \mathds{R} \times [0, \infty) \times \mathds{S}^2$, $M^- := \mathds{R} \times (-\infty, 0] \times \mathds{S}^2$ and where $r_\pm$ are given as inverses of the functions defined in Eq. \eqref{eq:rinverse}. The so-called thin shell function is given by $\eta: M \to \mathds{R}$, locally defined by $\eta \circ \psi_\pm^{-1} := \tilde{\eta}$. So we get a global continuous metric defined by
\be
g := (\Theta \circ \eta) ~ g^+ + \left( \Theta \circ (- \eta) \right) ~ g^-,
\ee
where $\Theta$ is the Heaviside distribution, so one has to see this definition in sense of distributions (we will recall this afterwards if you do not know the definition of distributions on arbitrary manifolds). To define the time orientation we state $\partial_{t_+}$ as a future-directed vector field.
\label{def:thinshellwormhole}
\end{definition}

\begin{remark}
\leavevmode\newline
Observe that the metric $g$ is in general only continuous by Eq. \eqref{eq:jump} due to the different sign there; moreover, on $M\setminus \eta^{-1}(\{0\}) = \mathds{R} \times \left(\mathds{R}\setminus \{0\} \right) \times \mathds{S}^2$ the metric is smooth, around $\eta^{-1}(\{0\}) = \mathds{R} \times \{0\} \times \mathds{S}^2$, the so-called thin shell, the metric is only continuous and has a kink there. Thence, the first derivatives will have a jump at the thin shell and the second ones are somehow proportional to the delta distribution.
\newline The metric also lives in the set of distributional tensors of rank $(0,2)$ which we will now shortly introduce because we want to calculate the curvature tensors and thence we need a distributional derivative, following \cite{VickerAnfang} and \cite{Sobolev}.
\label{remark:metric}
\end{remark}

\begin{definition}[Distributions on manifolds]
\leavevmode\newline
The set of scalar distributions $\mathcal{D}'(M)$ on a $n$-dimensional orientable and connected manifold $M$ contains all linear and continuous\footnote{Continuity is defined locally by using the standard sense of continuity of distributions; see e.g. \cite{lieb2001analysis}, Chapter 6 for more details. We will not need the technicalities of continuity.} maps $L: \Omega^n_c(M) \to \mathds{R}$, where $\Omega^n_c(M)$ is the set of $n$-forms with compact support, so the index $c$ will always denote compact support. The distributional tensors of rank $(r,s)$, $r,s \in \mathds{N}_0$, are then defined by
\ba
\mathcal{D}'(M)^r_s &:= \mathcal{D}'(M) \otimes_{C^\infty(M)} \mathcal{T}^r_s(M)
\cong \mathrm{L}_{C^\infty(M)}\left( \Omega^1(M)^r \times \mathfrak{X}(M)^s; \mathcal{D}'(M) \right) \nonumber \\
 & \cong \mathrm{L}_{C^\infty(M)}\left( \mathcal{T}^s_r(M); \mathcal{D}'(M) \right),
\ea
where $\mathrm{L}_{C^\infty(M)}(A; B)$ denotes the set of maps from $A$ to $B$ (suitable sets) which are homogeneous and linear over $C^\infty(M)$, $\mathcal{T}^r_s(M)$ is the set of sections of usual smooth tensors with rank $(r,s)$, $\Omega^1(M)$ is the set of 1-forms and $\mathfrak{X}(M)$ denotes the set of vector fields. For more details and for the proofs of these isomorphisms see \textit{e.g.} \cite{Sobolev}, subsection 2.1. One does not necessarily need to understand this notation, basically one can think of distributional tensors as tensors whose components are now in general distributions and not only smooth functions like usual, \textit{i.e.} allowing \textit{e.g.} the Delta distribution as some part of the tensor in coordinates. The Lie derivative $\mathcal{L}$ of $T \in \mathcal{D}'(M)^r_s$ is defined by
\ba
&\forall X \in \mathfrak{X}(M): ~ \forall t \in \mathcal{T}^s_r(M): ~ \forall w \in \Omega^n_c(M): ~ \\
&\left(\mathcal{L}_X T\right)(t)(w)
:= - T(\mathcal{L}_Xt)(w) - T(t)(\mathcal{L}_Xw)
\ea
where the outcome is again a distributional tensor such that one can take arbitrarily many derivatives without leaving the framework; here $T(t)$ is simply given by the contraction of dual tensors and $T(t) \in \mathcal{D}'(M)$ by definition, so $T(t)(w)$ means $(T(t))(w)$. The definition is motivated by partial integration, thus the sign change, and the product rule (for a motivation, think of $T$ as a smooth tensor and then the action of $T(t)$ on $w$ is simply given by integration over the spacetime).
\label{def:distribution}
\end{definition}

We want to study the physical existence of a wormhole as defined in Def. \ref{def:thinshellwormhole}, especially (as mentioned in the introduction) we want to study the so-called null energy condition (NEC) as defined as follows (as defined in \cite{Visser}). The indices run from 0 to 3 as usual.

\begin{definition}[NEC]
\leavevmode\newline
Let $T \in \mathcal{T}^0_2(M)$ (for us always the stress-energy tensor). Then we define
\ba
T \text{ satsfies the NEC}
&:\Leftrightarrow
\forall X \in \mathfrak{X}(M) \text{ with } g(X,X) \equiv 0:
T(X,X) \geq 0,
\ea
which basically means that the mass-energy density measured by a lightlike observer $X$ shall be non-negative. In local coordinates this definition is written as
\ba
\forall X \in \mathfrak{X}(M) \text{ with } X^\mu X^\nu g_{\mu \nu} \equiv 0:
X^\mu X^\nu T_{\mu \nu} \geq 0.
\ea
\end{definition}

As shown in \cite{Visser} (Chapter 14 ff.) the NEC is violated for any thin shell wormhole when one uses the standard Einstein field equations which is the reason why one says that wormholes consist of exotic matter. But we will instead use the formalism of the $F(R)-$gravity\footnote{$R$ is the scalar curvature and $F$ is some polynomial.} (following \cite{f(R)}) and then we study the NEC in this generalized version of gravity. Up to this point we assume that regularity of the metric $g$ and scalar curvature $R$ is no problem.

\begin{definition}[$F(R)$-gravity]
\leavevmode\newline
Let $T \in \mathcal{T}(M)^0_2$ be the stress-energy tensor. Then the generalized Einstein field equation (in sense of $F(R)$-gravity) is defined as ($\kappa := 8 \pi$)
\ba
\kappa T
=& ~
(F' \circ R) ~ \mathrm{Ric} - \frac{F \circ R}{2} ~ g
+ (F'' \circ R) \cdot \left(\square (R) ~ g - \nabla^2R\right) \nonumber \\
&
+ (F''' \circ R) \cdot \left( || \mathrm{grad}(R) ||^2 ~ g - \nabla R \otimes \nabla R \right),
\label{eq:F(S)}
\ea
where (using the Einstein sum convention and $g^{ij}$ are the components of the inverse of $g$)
\bas
&\square R := \frac{1}{\sqrt{|g|}} ~ \partial_i \left( \sqrt{|g|} ~ g^{ij} ~ \partial_jR \right) \text{ locally and Ric is the Ricci curvature tensor}, \\ 
&F: \mathds{R} \to \mathds{R} \text{ by } F(x) := \sum_{k= 0}^{N} \alpha_k x^k, N \in \mathds{N}, \alpha_k \in \mathds{R}, \alpha_0 = 2 \Lambda, \alpha_1 = 1, \\
&\nabla^2: \mathfrak{X}(M) \times \mathfrak{X}(M) \times \mathcal{T}(M)^r_s \to \mathcal{T}(M)^r_s \text{ by } \nabla^2_{X,Y}T := \nabla_X \nabla_Y T - \nabla_{\nabla_X Y} T, \\
& \mathrm{grad}(R) := g^{ij} \partial_i R ~ \partial_j \text{ locally}, ~|| \mathrm{grad}(R) ||^2 := \mathrm{grad}(R)(R) = g^{ij} \partial_iR ~ \partial_j R \text{ locally}, \\
&\text{where } \nabla \text{ is the Levi-Civita connection and the prime denotes the derivative}.
\eas
In local coordinates one has (the Christoffel symbols are denoted by $\Gamma^\alpha_{~\mu\nu}$)
\ba
\kappa T_{\mu \nu}
=& ~
\sum_{k=1}^{N} k\alpha_k R^{k-1} ~ \mathrm{Ric}_{\mu \nu} - \sum_{k=0}^N \frac{\alpha_k}{2} R^k ~ g_{\mu \nu} \nonumber \\
& + \sum_{k=2}^N k(k-1) \alpha_k R^{k-2} \cdot \left(\frac{1}{\sqrt{|g|}} ~ \partial_i \left( \sqrt{|g|} ~ g^{ij} ~ \partial_jR \right) ~ g_{\mu \nu} - \nabla_\mu \nabla_\nu R
+ \Gamma^\alpha_{~\mu \nu} \nabla_\alpha R \right) \nonumber \\
&
+ \sum_{k=3}^N k(k-1)(k-2) \alpha_k R^{k-3} \cdot \left( g^{\alpha \beta} \partial_\alpha R ~ \partial_\beta R ~ g_{\mu \nu} - \nabla_\mu R ~ \nabla_\nu R \right).
\ea
\label{def:F(R)}
\end{definition}

\begin{remark}
\leavevmode\newline
If one is interested into the derivation of this equation using variational calculus see \textit{e.g.} \cite{f(R)} (and then apply the chain rule to the given generalized Einstein field equation). By defining $\alpha_k := 0$ for all $k \geq 2$ one gets clearly back the Einstein field equations with the cosmological constant $\Lambda$.
\newline As one can see we can neglect the terms proportional to the metric because we want to check the NEC which means that we look at any lightlike vector field $X$, \textit{i.e.} $g_{\mu\nu}X^\mu X^\nu \equiv 0$.
\end{remark}

Since we already have the metric we can calculate all the arising curvature terms and insert them into Eq. \eqref{eq:F(S)} to study the NEC. But as explained in Rem. \ref{remark:metric} the metric will be only continuous on the thin shell $\eta^{-1}(\{0\})$. Thence, the scalar curvature will be proportional to a delta distribution with support on $\eta^{-1}(\{0\})$ such that we need the distributional geometry to describe our situation. Moreover, we will need a framework in which we can define a multiplication between any distributions because we have a highly non-linear situation, \textit{e.g.} for quadratic $F$ we need to define the square of the delta distribution to evaluate $F(R)$. Therefore we will now introduce the needed mathematics for this.

\section{Mathematical situation}

Although there are multiplications in the definition of the curvature terms we can calculate them in the distributional framework because the thin shell is a submanifold of codimension 1 in our case; see \textit{e.g.} \cite{VickerAnfang}, Theorem 2.2 and its discussion. But beware that one can not take the divergence of the arising stress-energy tensor given by the standard Einstein field equations (so not even in the situation of standard general relativity) as also mentioned in \cite{VickerAnfang}; basically the problem there is that one has to define the multiplication of the Christoffel symbols with the stress-energy tensor due to the definition of the Levi-Civita connection. But as explained the stress-energy tensor will be proportional to the delta distribution in standard general relativity (see also \cite{Visser}, Chapter 14) with support on the thin shell while the Christoffel symbols are discontinuous on the thin shell. The delta distribution with support at the thin shell can clearly be multiplied with tensors which are continuous around the thin shell but one can not define the multiplication with a tensor which is discontinuous around the thin shell; this can be roughly motivated by thinking of $f(x) \delta(x) = f(0) \delta(x)$. Which value for $f(0)$ one has to choose when $f$ has a jump at zero? This is problematic. One could try to solve this by defining the multiplication between the Heaviside and delta distribution how it is done in \cite{Visser} (Chapter 14, Section 14.4) by simply choosing some value with respect to the underlying physics because discontinuous functions can be viewed as a sum of suitable Heaviside distributions. But this is not rigorous as proven in \cite{VickerVorarbeit}, Section 2, because under natural assumptions one can easily see that then the delta distribution would have to be zero which can clearly not be wanted. We will see that one can define this multiplication using the Colombeau algebra without violating the physical statements in \cite{Visser} and the mathematical ones in \cite{VickerVorarbeit}.

But before, we have to define the delta distribution with support on the thin shell to discuss such things.

\subsection{Defining compositions with the delta distribution}

As already mentioned one can calculate every curvature term in our situation using the distributional framework, moreover, this can be done in the more general case of a thin shell: Let $\eta \in C^\infty(M)$ such that $||\mathrm{grad}(\eta)||^2(p) = \mathrm{grad}(\eta)(\eta) \stackrel{\mathrm{locally}}{=} g^{ij} \partial_j(\eta) ~ \partial_i (\eta) \neq 0$ for all points $p$ on an arbitrary connected orientable manifold $M$ (especially, 0 is a regular value,\footnote{This implies that $\eta^{-1}([0,\infty))$ is a smooth embedded submanifold of $M$ with boundary $\eta^{-1}(\{0\})$.} \textit{i.e.} $(\partial_i(\eta)(0))_{i=0}^3 \neq 0$ and the thin shell is not lightlike) where the metric $g$ has the form
\ba
g = g^+  ~ (\Theta \circ \eta) + g^- ~ (\Theta \circ (-\eta)),
\ea
where $g^+$ and $g^-$ are smooth metrics on $\eta^{-1}([0,\infty))$ and $\eta^{-1}((-\infty,0])$, respectively, each equipped with the same index $(-+++)$ (extended smoothly over the whole manifold) and such that $\left.g^+\right|_{\eta^{-1}(\{0\})} = \left.g^-\right|_{\eta^{-1}(\{0\})}$; so $g$ is continuous on $M$ but smooth on $M \setminus \eta^{-1}(\{0\})$. To calculate the curvature tensors one can then basically use the distributional derivatives like usually. We will skip these straightforward calculations but one term will be complicated, one needs a definition for the first derivative of the Heaviside distribution, denoted by $\delta \circ \eta$, \textit{i.e.} what one calls a composition of the delta distribution $\delta$ with $\eta$.

This can be easily done as follows, assuming that $\mathrm{grad}(\eta)$ is smooth;\footnote{In our case $\mathrm{grad}(\eta)$ is smooth and not only continuous such that one can still use $\Omega_c^n(M)$ as domain of definition and so one can still use standard definitions. In general one can only assume that $\mathrm{grad}(\eta)$ is continuous since the metric itself is only continuous. It is straightforward to handle this by using weak derivatives or using Gaußian normal coordinates to build a differentiable structure in which $\eta$ becomes a coordinate.} normally this is done by approximating the delta distribution, but we will show an easier way for which it is much easier to see that it is well-defined. Moreover, using the Colombeau algebra we can easily see later that both approaches are equivalent.

\begin{definition}[Compositions of the delta distribution]
\leavevmode\newline
$\delta \circ \eta: \Omega_c^n(M) \to \mathds{R}$ is defined by
\ba
\delta \circ \eta
:=
\mathcal{L}_{\frac{\mathrm{grad}(\eta)}{|| \mathrm{grad}(\eta) ||^2}}(\Theta \circ \eta).
\ea
\label{def:compwitheta}
\end{definition}

\begin{remark}
\leavevmode\newline
The idea of this definition is simply to define the delta distribution as a first derivative of the Heaviside distribution along the direction of $\frac{\mathrm{grad}(\eta)}{|| \mathrm{grad}(\eta) ||^2}$ because then one can expect that there will be no contribution of the inner derivative since $\frac{\mathrm{grad}(\eta)}{|| \mathrm{grad}(\eta) ||^2}$ acting on $\eta$ gives 1. See also Remark \ref{rem:motivdelta} for a comparison with a classical formula of $\delta \circ \eta$. 
\end{remark}

This definition can be simplified by the following theorem.

\begin{theorem}[More explicit formula for $\delta \circ \eta$]
\ba
\forall w \in \Omega_c^n(M): ~ (\delta \circ \eta)(w) = - \int_{\eta^{-1}(\{0\})} \frac{1}{|| \mathrm{grad}(\eta) ||^2} ~ i_{\mathrm{grad}(\eta)} w,
\ea
where the orientation of $\eta^{-1}(\{0\})$ is naturally given as the boundary of $\eta^{-1}([0, \infty))$ whose orientation itself is naturally induced by the orientation of $M$. Here $i$ means the interior product of differential forms.
\label{thm:deltcompetaexp}
\end{theorem}

\begin{proof}
\leavevmode\newline
Using Def. \ref{def:compwitheta} one gets for $X := 1/||\mathrm{grad}(\eta)||^2 ~ \mathrm{grad}(\eta)$ and $w \in \Omega_c^n(M)$
\begin{align*}
(\delta \circ \eta)(w)
&= -(\Theta \circ \eta)(\mathcal{L}_Xw)
= - \int_{\eta^{-1}([0,\infty))} \underbrace{\mathcal{L}_X w}_{\stackrel{w \text{ a } n\text{-form}}{=} (\mathrm{d} \circ i_X)w}
= - \int_{\eta^{-1}([0,\infty))} \mathrm{d}\left(i_Xw\right) \\
&\stackrel{\makebox[0mm][c]{\footnotesize \text{Stokes}}}{=} ~  - \int_{\eta^{-1}(\{0\})} i_X w,
\end{align*}
where we used that $w$ has compact support (whose compactness is not affected by $\mathrm{d}$ or $i_X$) and that $\eta^{-1}([0,\infty))$ is closed in $M$ when we applied Stokes' Theorem. $\eta^{-1}(\{0\})$ is equipped with its natural induced orientation as the boundary of $\eta^{-1}([0,\infty))$ by Stokes' Theorem.
\end{proof}

\begin{remark}
\leavevmode\newline
To motivate these definitions let us assume everything has appropriate regularities (\textit{e.g.} smooth) such that one can do everything as one would like.
\newline The basic motivation of these definitions is to define the delta distribution as the first derivative of the Heaviside distribution. So let us look at $M= \mathds{R}^n$. Then for $w \in \Omega^n_c(\mathds{R}^n)$ we write $w = f \mathrm{d}x$ for $f \in C^\infty_c(M)$ such that we get ($\theta$ is the standard Heaviside function representing the Heaviside distribution $\Theta$)
\begin{align*}
(\Theta \circ \eta) (w) &= \int_{\mathds{R}^n} (\theta \circ \eta) ~ f ~ \mathrm{d}x = \int_ {\eta^{-1}([0,\infty))} f ~ \mathrm{d}x
\end{align*}
and by calculating the first distributional derivative $\partial_i(\Theta \circ \eta)$ one gets (using the integral notation with $\partial_i(\theta \circ \eta)$ as in physics)
\begin{align*}
\partial_i(\Theta \circ \eta)(w) &= \int_{\mathds{R}^n} \partial_i(\theta \circ \eta) ~ f ~ \mathrm{d}x = - \int_{\eta^{-1}([0, \infty))} \partial_i f ~ \mathrm{d}x
\stackrel{\text{Gauss}}{=} - \int_{\eta^{-1}(\{0\})} f ~ \nu_i ~ \mathrm{d}\sigma,
\end{align*} 
where $\mathrm{d}\sigma$ is a canonical surface measure and $\nu_i$ is the $i$-th component of the unit normal vector pointing outwards, \textit{i.e.} $\nu_i := - \partial_i\eta/|\nabla \eta|$ as usual. Using the physic's notation $\partial_i(\theta \circ \eta) = (\delta \circ \eta) ~ \partial_i \eta$ one can then write
\begin{align*}
\int_{\eta^{-1}(\{0\})} f ~ \frac{\partial_i \eta}{\left| \nabla \eta \right|} ~ \mathrm{d}\sigma
&= \int_{\mathds{R}^n} (\delta \circ \eta) ~ \partial_i \eta ~ f ~ \mathrm{d}x.
\end{align*}
Since $i$ was arbitrary (or replacing $\partial_i$ in the same way like we replaced $X$ in the previous definition) one could define $\delta \circ \eta$ by comparing this with the other physics' notations $\delta(f) = \int_{\mathds{R}^n} \delta ~ f \mathrm{d}x$ and $(\delta\circ\eta)(f) = \int_{\mathds{R}^n} (\delta\circ\eta) ~ f \mathrm{d}x$, so
\begin{equation}
(\delta \circ \eta)(f) := \int_{\eta^{-1}(\{0\})} \frac{f}{\left| \nabla \eta \right|} ~ \mathrm{d}\sigma.
\label{eq:deltacompieta}
\end{equation}
Also compare the theorem with Eq. \eqref{eq:deltacompieta}: Fix a volume form dvol and write $w = f ~ \mathrm{dvol}$ for $f \in C^\infty_c(M)$, then the canonical volume form d$\sigma$ on the boundary of an oriented manifold is the contraction $i_\nu\mathrm{dvol}$ where $\nu$ is the unit normal vector pointing outside. If we assume that we have a standard Riemannian metric, \textit{i.e.} $g$ is a scalar product on each fibre of $\mathrm{T}M$ such that we can take the square root of $|| \mathrm{grad}(\eta) ||^2$, then this means in this case that
\begin{align*}
\mathrm{d}\sigma = i_{-\frac{\mathrm{grad}(\eta)}{||\mathrm{grad}(\eta)||}}\mathrm{dvol}.
\end{align*}
Inserting this in the result of the theorem one gets
\begin{align*}
(\delta \circ \eta)(w) = \int_{\eta^{-1}(\{0\})} \frac{f}{|| \mathrm{grad}(\eta) ||} ~ \mathrm{d}\sigma,
\end{align*}
which is clearly Eq. \eqref{eq:deltacompieta} in the Euclidean case.
\newline Moreover, let us compare this result with the most known result when one has finitely many zeroes of $\eta$ (and we restrict for simplicity again to the Euclidean case, so we can take the standard Euclidean volume form and metric), where we denote the zeroes by $x_1, \dots, x_k$. Then $\eta^{-1}(\{0\})$ consists only of points and $\int_{\eta^{-1}(\{0\})} \mathrm{d}\sigma$ is just the counting measure (to see this, beware in the one-dimensional case that one has a sign at any point of the evaluation $\int_{\eta^{-1}(\{0\})} \dots = \dots |_{\eta^{-1}(\{0\})}$ like in the fundamental theorem of calculus in one dimension, but this sign is cancelled by the arising signs in $\mathrm{d}\sigma = i_{-\frac{\mathrm{grad}(\eta)}{||\mathrm{grad}(\eta)||}}\mathrm{dvol}$, \textit{i.e.} the sign of $- \mathrm{grad}(\eta)$, which can easily be seen). Thus,
\bas
(\delta \circ \eta)(f) = \sum_{l=1}^k \frac{f(x_l)}{|\nabla\eta(x_l)|},
\eas
which is the standard known formula.
\newline Finally, the result of the theorem is more general since it allows continuous pseudo-Riemannian metrics and it is valid for any orientation. Additionally let us mention that both Definition \ref{def:compwitheta} and Theorem \ref{thm:deltcompetaexp} could thence clearly not work with a normal vector which could be lightlike as for lightlike submanifolds; one might then try to find another vector field with $\mathcal{L}_X(\eta)=1$ with which one takes the Lie derivative of $\Theta \circ \eta$.
\label{rem:motivdelta}
\end{remark}

One can easily see the following standard properties.

\begin{proposition}[Properties of $\delta \circ \eta$]
\leavevmode\newline
i) 
\ba
\delta \circ (-\eta) = \delta \circ \eta.
\ea
ii)
\ba
\forall X \in \mathfrak{X}(M): ~ 
\mathcal{L}_X(\Theta \circ \eta) = \mathcal{L}_X(\eta) ~ (\delta \circ \eta),
\ea
where the multiplication on the right hand side is defined as usual by $\Omega^n_c(M) \ni w \mapsto (\delta \circ \eta)(\mathcal{L}_X(\eta) ~ w)$ which is clearly well-defined since a distribution multiplied with a smooth function by this rule is still clearly a distribution.
\label{prop:deltcompetaexp}
\end{proposition}

\begin{proof}
\leavevmode\newline
\textit{i)}
\newline Let $w \in \Omega_c^n(M)$, then by using Thm. \ref{thm:deltcompetaexp}
\bas
(\delta \circ (-\eta))(w)
&=
- \int_{(-\eta)^{-1}(\{0\})} \frac{i_{-\mathrm{grad}(\eta)} w}{||\mathrm{grad}(-\eta) ||^2}
=
\int_{(-\eta)^{-1}(\{0\})} \frac{i_{\mathrm{grad}(\eta)} w}{|| \mathrm{grad}(\eta) ||^2} \\
&=
- \int_{\eta^{-1}(\{0\})} \frac{i_{\mathrm{grad}(\eta)} w}{|| \mathrm{grad}(\eta) ||^2}
= (\delta \circ \eta)(w)
\eas
because $(-\eta)^{-1}(\{0\}) = \eta^{-1}(\{0\})$ as sets but we also need its orientation to define the integral. The orientation for $\eta^{-1}(\{0\})$ is naturally given by $\eta^{-1}([0, \infty))$ (as a boundary) and likewise the orientation for $(-\eta)^{-1}(\{0\})$ is given by $(-\eta)^{-1}([0, \infty))$. Since $(-\eta)^{-1}([0, \infty)) = \eta^{-1}((-\infty,0])$ [which gives the opposite orientation on its boundary which is also $\eta^{-1}(\{0\})$] we not only know that $(-\eta)^{-1}(\{0\}) = \eta^{-1}(\{0\})$ as sets we also know that both sets do have opposite orientations such that one gets an overall sign change.
\newline \textit{ii)} \newline
$\forall X \in \mathfrak{X}(M)$:
\bas
(\mathcal{L}_X(\Theta \circ \eta))(w)
&=
- (\Theta \circ \eta)(\mathrm{d} i_X w)
\stackrel{\text{Stokes}}{=}
- \int_{\eta^{-1}(\{0\})} i_X w \\
&\stackrel{\makebox[0mm][c]{\footnotesize \text{Thm. \ref{thm:deltcompetaexp}}}}{=}
\quad ~ - \underbrace{\int_{\eta^{-1}(\{0\})} i_{X - \mathcal{L}_X(\eta) \frac{\mathrm{grad}(\eta)}{||\mathrm{grad}(\eta)||^2}} w}_{=0}
+ (\delta \circ \eta)(\mathcal{L}_X(\eta) ~ w),
\eas
by $\left( X - \mathcal{L}_X(\eta) \frac{\mathrm{grad}(\eta)}{||\mathrm{grad}(\eta)||^2} \right)(\eta) = 0$, \textit{i.e.} $\left.\left(X - \mathcal{L}_X(\eta) \frac{\mathrm{grad}(\eta)}{||\mathrm{grad}(\eta)||^2}\right)\right|_{\eta^{-1}(\{0\})}$ is a vector field tangent to the thin shell (since $\mathrm{grad}(\eta)$ is linear independent, \textit{i.e.} non-zero, such that one can extend it to a coordinate frame locally around the thin shell onto which $\mathrm{grad}(\eta)$ is clearly normal by the definition of the thin shell; so $\eta$ basically describes a normal coordinate to the thin shell while the other coordinates can be chosen to be tangent at the thin shell). So, the first summand has an integration over $\eta$ as a coordinate in it and, thus, it vanishes because the integral domain over $\eta$ is a set of measure zero, namely $\{0\}$.
\end{proof}

\subsection{Colombeau algebra}

Using these results one can calculate the curvature terms in the distributional framework in a straightforward manner. We now turn to the generalized algebra of the distributional framework such that one can define multiplications between any distribution. The framework we will use is the Colombeau algebra, following \cite{VickerAnfang}, \cite{Sobolev}, \cite{GrundlagenI} and \cite{GrundlagenII}.

The basic idea is to use that any distributional tensor $L \in \mathcal{D}'(M)^r_s$ can be approximated by a sequence of smooth tensors $(\mathcal{L}_{\varepsilon})_{\varepsilon \in (0,1]} \subset \mathcal{T}^r_s(M)$ (see \textit{e.g.} \cite{lieb2001analysis} for the case of scalar distributions in $\mathds{R}$ and the references listed above for the more general case) such that (recall Def. \ref{def:distribution}; $\cdot$ denotes the contraction of dual tensor fields)
\ba
\forall t \in \mathcal{T}^s_r(M): ~ \forall w \in \Omega_c^n(M): ~ 
\lim_{\varepsilon \to 0} \int_M (\mathcal{L}_\varepsilon \cdot t) ~ w
&= 
L(t)(w),
\label{eq:approxintegral}
\ea
where $\mathcal{L}_\varepsilon = L * \rho_\varepsilon$, \textit{i.e.} the convolution with a suitable delta net $\rho_\varepsilon$ which is defined as an admissible mollifier (approximations especially of the identity of compact sets and step functions, see also the following Lemma \ref{lem:admismoll}). The convolution of this mollifier with a tensor is defined by mollifying the components of the tensor by fixing some coordinate frame where we assume that one has \textit{e.g.} global coordinates. But there is also a generalized definition of a convolution for arbitrary manifolds, see \textit{e.g.} the remarks and the references in the section of prerequisites in \cite{Schwarzschild}. In our case we will be able to use the standard definition due to the fact that we have global coordinates. With $\mathcal{O}$ we denote the Big O notation.

\begin{lemma}[Existence of admissible mollifier]
\leavevmode\newline
There exists a net $\left(\psi_\varepsilon\right)_{\varepsilon \in (0,1]} \subset C^\infty_c(\mathds{R}^n)$ with the properties
\ba
\forall \varepsilon \in (0,1]: & ~\mathrm{supp}\left(\psi_\varepsilon\right) \subset B_1(0), \label{cond:i} \\
\forall \varepsilon \in (0,1]: &~ \int \psi_\varepsilon(x) \mathrm{d}x = 1, \label{cond:ii} \\
\forall \alpha \in \mathds{N}^n_0: ~ \exists p \in \mathds{R}: & \sup_{x \in \mathds{R}^n} \left| \partial^\alpha \psi_\varepsilon(x) \right| = \mathcal{O}\left( \varepsilon^{-p} \right) (\varepsilon \to 0), \label{cond:iii}\\
\forall j \in \mathds{N}: ~ \exists \varepsilon_0 \in (0,1]: ~ \forall 1 \leq |\alpha|\leq j:~ \forall \varepsilon \leq \varepsilon_0: & ~\int x^\alpha ~ \psi_\varepsilon ~ \mathrm{d}x = 0, \label{cond:iv} \\
\forall \nu > 0: ~ \exists \varepsilon_0 \in (0,1]: ~ \forall \varepsilon \leq \varepsilon_0: &~ \int \left|\psi_\varepsilon\right| \mathrm{d}x \leq 1 + \nu.
\label{cond:v}
\ea
I.e.,
\ba
\rho_\varepsilon := \frac{1}{\varepsilon^n} ~ \psi_\varepsilon\left( \frac{\cdot}{\varepsilon} \right)
\ea
is a delta net which is moderate (condition \eqref{cond:iii}), has finitely many vanishing moments (condition \eqref{cond:iv}) and the negative part has arbitrary small L$^1$-norm (condition \eqref{cond:v}).
\label{lem:admismoll}
\end{lemma}

\begin{proof}
See the appendix of \cite{Sobolev}.
\end{proof}

\begin{remark}
\leavevmode\newline
It is not important to understand these mathematical technicalities. If you are a physicist who does not work that intensive with mathematics then you still have seen them at least once in your studies. You can think of these mollifiers/delta nets $\rho_\varepsilon$ as an approximation of the delta distribution like the typical sequences you may have seen once; this definition covers some specific sort of such sequences. Simply think of them as some approximation of the delta distribution. Condition \eqref{cond:iii} means that the singularities in the sequence and its derivatives are just of type $\varepsilon^{-p}$, often the height of $\rho_\varepsilon$ at 0 increases with smaller $\varepsilon$ by $1/\varepsilon$ while the "width" decreases by $\varepsilon$ due to condition \eqref{cond:ii}.\footnote{Think for these two conditions of the typical rectangle sequence approximating the delta distribution.} Conditions \eqref{cond:ii} and \eqref{cond:v} only preserves that the negative part is vanishing and that the integral is 1, \textit{i.e.} $\delta$ is a positive distribution ("positive singularity") and $\int \delta = 1$, the normalization condition. Condition \eqref{cond:i} preserves that $\rho_\varepsilon(x)$ converges to 0 for $x \neq 0$, a physicist would write $\delta(x) = 0$ for all $x \neq 0$.
\newline Condition \eqref{cond:iv} is very essential mathematically for the Colombeau algebra, this condition will lead to the fact that our very general sense of multiplication will preserve the multiplication of smooth tensors; see the cited references for the Colombeau algebra.
\end{remark}

Replacing any arising distribution with such a sequence one has a natural way to define multiplications by the pointwise multiplication of these sequences. But of course one wants to generalize the multiplication and the tensor product, \textit{i.e.} one wants to embed the distributional framework into a differential algebra generalizing addition, derivative and tensor product. There is an incompatibility result of Schwartz which states that this can not be done for the multiplication of $C^k$-tensors (see \cite{VickerVorarbeit}, section 2), $k \in \mathds{N}_0$, but for smooth tensors this is possible. By the sequence approach this is not yet done since the convolution is clearly not a homomorphism with respect to the tensor product, \textit{i.e.} $f * (g \cdot h) \neq (f * g) \cdot (f * h)$ in general, so a multiplication of smooth tensors is not necessarily preserved. But a smooth tensor can also be approximated by a constant sequence since the tensor is already smooth and for constant sequences it is clear that the standard tensor product between smooth tensors will be generalized. Thus, the idea is to identify both sequence approaches for smooth tensors, the constant sequence and the sequence using the convolution, by doing a suitable quotient structure (w.r.t. addition). There is no place to discuss this completely in this paper, for a more complete picture see \textit{e.g.} \cite{Sobolev} and the references therein. However, to calculate things it is not really important to understand the rigorous definition. Keep simply the idea of using a sequence $u_\varepsilon$ of smooth tensors to approximate a distribution $u$.

Our approach of defining the Colombeau algebra will be dependent on the admissible mollifier and on coordinates; if you want to see a mollifier and coordinate independent version, the full Colombeau algebra, see \cite{GrundlagenI} and \cite{GrundlagenII}. We will only state the definition of the Colombeau algebra using a fixed admissible mollifier $\rho_\varepsilon$ (like in \cite{Sobolev} or in \cite{Schwarzschild}) and an arbitrary norm $|| \cdot ||$ for tensors induced by a standard metric on manifolds.\footnote{The following definitions will not depend on this metric, see \textit{e.g.} \cite{GrundlagenII} or \cite{Sobolev}.}

\begin{definition}[Colombeau algebra]
\leavevmode\newline
Let $(u_\varepsilon)_\varepsilon$ be a sequence of smooth tensors (where $\mathcal{O}$ denotes the big O notation).
\newline\textit{i)} (Negligible tensors)
\bas
(u_\varepsilon)_\varepsilon \in \hat{\mathcal{N}}(M)^r_s 
~ :\Leftrightarrow ~ &
\forall K \subset\subset M: ~ \forall k \in \mathds{N}_0: ~ \forall m \in \mathds{N}_0: ~ \forall X_1, \dots, X_k \in \mathfrak{X}(M): ~ \\
&\sup_{x \in K} || \mathcal{L}_{X_1} \dots \mathcal{L}_{X_k} u_\varepsilon(x) || = \mathcal{O}(\varepsilon^{m}) \text{ as } \varepsilon \to 0
\eas
\textit{ii)} (Moderate tensors)
\bas
(u_\varepsilon)_\varepsilon \in \hat{\mathcal{E}}_m(M)^r_s
~ :\Leftrightarrow ~ &
\forall K \subset\subset M: ~ \forall k \in \mathds{N}_0: ~ \exists N \in \mathds{N}: ~ \forall X_1, \dots, X_k \in \mathfrak{X}(M): ~ \\
&\sup_{x \in K} || \mathcal{L}_{X_1} \dots \mathcal{L}_{X_k} u_\varepsilon(x) || = \mathcal{O}\left(\varepsilon^{-N}\right) \text{ as } \varepsilon \to 0
\eas
\textit{iii)} (Colombeau algebra)
\bas
\hat{\mathcal{G}}(M)^r_s := \left.\hat{\mathcal{E}}_m(M)^r_s\middle/\hat{\mathcal{N}}(M)^r_s\right.,
\eas
and for all $[(u_\varepsilon)_\varepsilon], [(v_\varepsilon)_\varepsilon] \in \hat{\mathcal{G}}(M)^r_s$, $X \in \mathfrak{X}(M)$ we define
\bas
[(u_\varepsilon)_\varepsilon] + [(v_\varepsilon)_\varepsilon]
&:=
[(u_\varepsilon + v_\varepsilon)_\varepsilon], &
[(u_\varepsilon)_\varepsilon] \otimes [(v_\varepsilon)_\varepsilon]
&:=
[(u_\varepsilon \otimes v_\varepsilon)_\varepsilon], &
\mathcal{L}_X[(u_\varepsilon)_\varepsilon]
&:=
[(\mathcal{L}_X u_\varepsilon)_\varepsilon].
\eas
\label{def:Colombeau}
\end{definition}

\begin{remark}
\leavevmode\newline
In \cite{Sobolev} (especially also in the appendix) it is argued that this gives a generalized differential algebra such that one gets a generalized distributional framework now equipped with a tensor product generalizing the tensor product of smooth tensors; for the full Colombeau algebra, the mollifier and coordinate independent one, this is done in \cite{GrundlagenI} and \cite{GrundlagenII}.
\newline The injective embedding of distributions into the Colombeau algebra will be denoted by $\iota$, so for $L \in \mathcal{D}'(M)^r_s$ we have $\iota(L) = [(\mathcal{L}_\varepsilon)_\varepsilon]$ as defined in Eq. \eqref{eq:approxintegral}. For simplicity one often drops all the brackets, only writing $u_\varepsilon$ for elements of the Colombeau algebra and depending on the context one means with $\iota(L)$ either the sequence in sense of Colombeau (so the equivalence class without writing the brackets) or the sequence given by the convolution with a delta net.
\end{remark}

Since this algebra is mainly motivated by Eq. \eqref{eq:approxintegral} one could ask (especially after doing multiplications of distributions) what happens when one takes again $\varepsilon$ to 0. But the outcome does in general not converge, like the square of the delta distribution, see \textit{e.g.} \cite{Multiplikation}, Section 6. But the difference of two elements of the Colombeau algebra can converge such that one could try to identify both elements in sense of their behaviour for $\varepsilon \to 0$, \textit{i.e.} in their somewhat distributional behaviour. This is called the association as defined in \cite{Sobolev}, Subsection 2.3. 

\begin{definition}[Association]
\leavevmode\newline
Let $[(u_\varepsilon)_\varepsilon]$ and $[(v_\varepsilon)_\varepsilon] \in \hat{\mathcal{G}}(M)^r_s$.
\bas
[(u_\varepsilon)_\varepsilon] \approx [(v_\varepsilon)_\varepsilon]
~:\Leftrightarrow~&
\forall w \in \Omega_c^n(M): ~ \forall t \in \mathcal{T}(M)^s_r: 
\lim_{\varepsilon \to 0} \int_M \left( \left( u_\varepsilon - v_\varepsilon \right) \cdot t \right) ~ w = 0,
\eas
we then say that $[(u_\varepsilon)_\varepsilon]$ and $[(v_\varepsilon)_\varepsilon]$ are associated to each other.
\end{definition}

\begin{remark}
\leavevmode\newline
One can easily check that this is an equivalence relation, especially this integral equality holds for two representative sequences of the same equivalence class.
\newline For $L, S \in \mathcal{D}'(M)^r_s$ it is clear by Eq. \eqref{eq:approxintegral} that $\iota(L) \approx \iota(S)$ implies $L=S$ and so $\iota(L) = \iota(S)$; hence the association generalizes the distributional equality. Basically one can expect that also every known multiplication rule in the distributional setting is generalized in the Colombeau algebra in sense of association. But one has to prove this, see \textit{e.g.} \cite{GrundlagenII}, Section 9, Corollary 9.13, for a proof for the multiplication between continuous tensors or for multiplications between a smooth tensor and a distribution. Moreover, in \cite{Sobolev} it is shown that the distributional formulation of general relativity (the part of it which can be formulated by distributions) is preserved under the level of association; also in \cite{GenRiem} it is shown how the Riemannian geometry in sense of Colombeau works and how it is generalized in the level of association. Basically it is shown in \cite{GenRiem} that for a continuous metric (like in our case) one can do everything as usual, so all the formulas of the curvature terms stay in their usual form. One only has to replace the corresponding terms with their approximating sequences.
\newline An example for an association (see also \cite{Multiplikation} for the easy proof) is $\iota(\Theta)^2 \approx \iota(\Theta)$ (but $\iota(\Theta)^2 \neq \iota(\Theta)$) like one expects this in the distributional framework. It can be easily seen that the association is stable under the Lie derivative and thus one can then derive $\iota(\delta) ~ \iota(\Theta) \approx \frac{1}{2} \iota(\delta)$ by differentiating both sides of $\iota(\Theta)^2 \approx \iota(\Theta)$ $\big($similar for $\delta \circ \eta$ and $\Theta \circ (\pm\eta)$; the sign does not matter, \textit{i.e.} $\iota(\delta \circ \eta) ~ \iota(\Theta \circ (\pm \eta)) \approx \frac{1}{2} \iota(\delta \circ \eta)$$\big)$. Recall that we mentioned at the beginning of this section that such rules are not rigorous in the distributional framework although the multiplication rule $\delta \Theta = \frac{1}{2} \delta$ is used in \cite{Visser} (Chapter 14, Section 14.4) to well-define the divergence of the stress-energy tensor. One can see or hope that one can use the results although they are not rigorously derived in the distributional framework because the used multiplication rule holds on the weaker level of association in the Colombeau algebra. So translating the situation into the Colombeau algebra one may get the same results (which are based on the divergence of the stress-energy tensor) on the level of association. But then rigorously derived.
\label{rem:association}
\end{remark}

\section{Generalized Non-negativity and NEC}

\subsection{Non-negativity in the Colombeau algebra}

Using the Colombeau algebra one can now define the stress-energy tensor by using Eq. \eqref{eq:F(S)} although the scalar curvature $R$ will be a delta distribution with support on the thin shell (see the beginning of the fifth section to see an explicit expression). But the stress-energy tensor will now depend on the index $\varepsilon$ and thence we have to clarify now what it actually means to study the NEC. We would like a definition which is independent of $\varepsilon$ and which generalizes the NEC in sense of distributions, \textit{i.e.} we should find a suitable generalized definition of non-negativity. Let us shortly recall the sense of non-negativity for distributions. In the following $M$ is again a connected orientable $n$-dimensional smooth manifold.

\begin{definition}[Non-negative distributions]
\leavevmode\newline
Let $w \in \Omega_c^n(M)$ and fix a volume form dvol then we can write $w = f ~ \mathrm{dvol}$ for a $f \in C^\infty_c(M)$ and we define
\bas
w \geq 0 &:\Leftrightarrow f \geq 0.
\eas
Then for $L \in \mathcal{D}'(M)$ we define
\bas
L \geq 0 &:\Leftrightarrow \forall w \in \Omega_c^n(M) \text{ with } w \geq 0: ~ L(w) \geq 0.
\eas
\label{def:nonnegdistr}
\end{definition}

A first idea of defining non-negativity for an element $[(u_\varepsilon)_\varepsilon]$ of the Colombeau algebra would be to say that $[(u_\varepsilon)_\varepsilon]$ if and only if $[(u_\varepsilon)_\varepsilon]$ is associated to an embedded non-negative distribution. As already mentioned we will have a delta distribution as a scalar curvature and in $F(R)$ one could get for example the square of the delta distribution in a term of the stress-energy tensor. But as explained in the previous section a square of the delta distribution will not converge to any distribution for $\varepsilon \to 0$ because it diverges then. Therefore we have to find a slightly more general definition of non-negativity to handle this divergence. The following theorem gives arise to a suitable sense of non-negativity where $(\cdot)_-$ denotes the negative part (defined with negative values and not with positive values as also often done in literature); in the same manner we denote the positive part by $(\cdot)_+$.

\begin{theorem}[Vanishing negative part of a non-negative distribution]
\leavevmode\newline
Let $0 \leq L \in \mathcal{D}'(M)$ be a non-negative distribution where non-negativity is defined with respect to a volume form $\mathrm{dvol}$ such that $\int_K \mathrm{dvol} > 0$ for all compact non-empty subsets $K$ with naturally induced orientation and of dimension $n$ (so not of measure zero).\footnote{\textit{I.e.} it is a volume form corresponding to orientation like $\mathrm{d}^nx$ corresponds to the standard orientation of $\mathds{R}^n$.} Then
\ba
\forall w \in \Omega_c^n(M): ~
\lim_{\varepsilon \to 0^+} \int_{M} (\iota(L))_- ~ w = 0,
\ea
writing as
\ba
(\iota(L))_- \approx 0.
\ea
\label{thm:geilestheorem}
\end{theorem}

\begin{remark}
\leavevmode\newline
One can clearly state a similar theorem for another definitions of non-negativity in the distributional framework (like using $-$dvol instead such that then $(\iota(L))_+ \approx 0$). But we will of course work with the standard sense of non-negativity, \textit{i.e.} using a volume form which is for the given orientation somewhat positive, especially because $\delta \circ \eta$ will be non-negative in the distributional framework with respect to this volume form.
\newline One will see in the proof that the converse statement of the theorem easily follows. \textit{I.e.} for a distribution $L \in \mathcal{D}'(M)$ with $(\iota(L))_- \approx 0$ follows that $L \geq 0$ with respect to a volume form with the properties as stated in the theorem.
\end{remark}

\begin{proof}
\leavevmode\newline
Let $0 \leq L \in \mathcal{D}'(M)$ be a non-negative distribution, then clearly for all $\Omega_c^n(M) \ni w$
\ba
L(w)
&=
\lim_{\varepsilon \to 0^+} \int_M \iota(L) ~ w
=
\lim_{\varepsilon \to 0^+} \left(
\int_{M} (\iota(L))_+ ~ w
+ \int_{M} (\iota(L))_- ~ w
\right).
\label{eq:measureapprox}
\ea
Beware: Do not split the limes to both summands since both diverge in general for any distribution, only for distributions of zero order this would work (see also later). In our case it is possible but to show this directly one have to argue very technical which is not easy to write down.
\newline But we will solve this a lot easier. To avoid these tedious technicalities it is much easier to use measure theory. Take any distribution $V$ of zero order, \textit{i.e.} for all compact subsets $K$ exists a constant $c_K \geq 0$ such that for all compact $n$-forms $w$ with support in $K$ holds
\bas
|V(w)| \leq c_K ~ \sup_{x \in K}|f(x)|,
\eas
where $w = f \mathrm{dvol}$ for $f \in C_c^\infty(M)$ and for a fixed volume form dvol (in our case a volume form dvol like stated in the theorem). By the very well known Riesz-Markov-Saks-Kakutani representation theorem one knows that $V$ can be identified uniquely with a locally finite regular complex/real-valued Borel measure $\nu$ such that\footnote{Hence we have a generalized measure, in our case a signed measure where negative values are allowed.} (using the same notation as in \cite{lieb2001analysis}, Theorem 6.22)
\bas
\forall w = f \mathrm{dvol} \in \Omega_c^n(M): ~
V(w) = \int_M f ~ \nu(\mathrm{dvol}).
\eas
By the important Jordan measure decomposition (also often called Hahn-Jordan measure decomposition) one can split such a measure into its negative and positive part, \textit{i.e.}
\bas
\nu = \nu_+ - \nu_-,
\eas
where $\nu_\pm$ are non-negative measures (with respect to the given orientation) which are singular to each other as measures. So one has a so-called Hahn decomposition of $M= P_+ \cup P_-$ with $P_+ \cap P_- = \emptyset$ and $P_\pm$ are measurable subsets such that for all measurable subsets $E \subset P_+$ one gets
$\nu_-(E)=0$ and $\nu_+(E) \geq 0$ and the same vice versa when one switches $+$ and $-$ (so one can see some sense of linear independence for them). This means that we can define the positive ($V_+$) and negative ($V_-$) part of a distribution of zero order as clearly well-defined distributions by
\bas
\forall w = f \mathrm{dvol} \in \Omega_c^n(M): ~
V_\pm(w) = \int_M f ~ \nu_\pm(\mathrm{dvol}),
\eas
such that $V = V_+ - V_-$; since we take dvol as defined in the theorem it is clear why one calls this positive and negative part (w.r.t. the given orientation).
\newline Let us now turn back to $L$. Since $L \geq 0$ one knows by \cite{lieb2001analysis}, Theorem 6.22, that it is a distribution of zero order, thus related to a measure as explained above, so we gain a splitting (with respect to our taken dvol)
\bas
L = L_+ - L_-,
\eas
moreover, (by the taken dvol) we get $L_- = 0$ by the non-negativity of $L$ (so its associated measure is non-negative and by the fact that the negative and positive part of the associated measure are singular to each other we get that the negative part has to vanish). When one can split the limes in Eq. \eqref{eq:measureapprox} to both summands then one has a clearly candidate for $L_-$, namely the integral sequence of the negative part of $\iota(L)$. When we can show this then we are done. Look at the sequence $\int_M \iota(L) ~ w$ for a fixed $w$. For a fixed $\varepsilon$ this is clearly a distribution of zero order since it is a regular distribution, so we have a sequence of signed measures which can be again split into a negative and positive parts of their associated measure; for this regular distribution the measure associated is clearly $\iota(L) ~ \mathrm{dvol}$ (again for the fixed dvol stated in the theorem) and the splitting has the typical standard form 
\bas
\iota(L) ~ \mathrm{dvol}
&=
\underbrace{(\iota(L))_+ ~ \mathrm{dvol}}_{\text{positive part of the measure}}
-
\underbrace{(-(\iota(L))_- ~ \mathrm{dvol})}_{\text{negative part of the measure}}.
\eas
So we have a measure sequence of signed measures which converge pointwise to another signed measure (associated to $L$) using Eq. \eqref{eq:measureapprox} (for the convergence to a distribution of zero order which is a measure). Since the negative and positive part of a signed measure are singular to each other we clearly get that the sequences of the positive and negative part of the measure sequence also converge pointwise to $L_+$ and $L_-$, respectively (see \cite{masse}, Chapter 2, if you do not know these properties of measure theory). Thus, their limits are pointwise (translated back to their distributional sequences)
\bas
\forall w \in \Omega_c^n(M): ~
L_\pm(w)
&=
\lim_{\varepsilon \to 0^+}
\int_{M} (\pm(\iota(L))_\pm) ~ w,
\eas
so we had shown that in case of a distribution of zero order\footnote{This is an important condition. For another distributions one could not get a signed measure in the limit and then one is not able to use measure theory at this convergence argument.} we can distribute the limes over $\varepsilon$ to integrals over the positive and negative part which can not hold for a general distribution as mentioned before. Otherwise by assuming that this can be done for a general distribution (so the partial limits exist) then the limits would clearly describe distributions which are non-negative/non-positive so they would be associated to measures and so the general distribution would be a sum of signed measures. Hence every distribution would be then a signed measure and it is clear that a distribution can not be a measure in general (take \textit{e.g.} the first derivative of the delta distribution) which contradicts the assumption.
\newline Moreover, by $L_- = 0$ (as discussed before) we get
\bas
\forall w \in \Omega_c^n(M): ~
L_-(w)
&= -
\lim_{\varepsilon \to 0^+}
\int_{M} (\iota(L))_- ~ w = 0.
\eas
\end{proof}

Basically this theorem tells us that a non-negative function has a vanishing negative part which one would naively always say. Thus, we can define now the generalized non-negativity.

\begin{definition}[Non-negativity in the Colombeau algebra]
\leavevmode\newline
Let $\left[ \left( u_\varepsilon \right)_\varepsilon \right] \in \hat{\mathcal{G}}(M)$.
\bas
\left[ \left( u_\varepsilon \right)_\varepsilon \right] \geq 0
~&:\Leftrightarrow~
((u_\varepsilon)_-)_\varepsilon \approx 0.
\eas
\label{def:nonnegCol}
\end{definition}

\begin{remark}
\leavevmode\newline
By the previous theorem one knows that on has a generalized definition of non-negativity and by the remark of the previous theorem we also know that (with respect to the suitable fixed volume form) the set of non-negative distributions does not increase by this definition such that this definition is not too general. Moreover, we can now look at terms such as the square of the delta distribution which is non-negative by this definitions since the negative part of a square is zero. Thus, by taking the negative part one can avoid the (positive) divergence of the squared delta distribution.
\newline This definition is clearly independent of the representative function by using $-(f+h)_- \leq -f_- -h_-$ for Borel measurable functions which can easily seen by
\bas
&&
f &= f_+ + f_-
&&\wedge
& |f| = f_+ - f_- \\
&\Rightarrow&
f_-
&=
\frac{f - |f|}{2}
\eas
and then use the upper triangle inequality.
\label{rem:nonneg}
\end{remark}

It is easy to see that this sense of non-negativity is stable under association.

\begin{corollary}[Non-negativity is stable under association]
\leavevmode\newline
Let $u$ and $v$ be in $\hat{\mathcal{G}}(M) := \hat{\mathcal{G}}^0_0(M)$ with $u \approx v$ where $u \geq 0$. Then we also get $v \geq 0$.
\label{cor:NECgleich}
\end{corollary}

\begin{proof}
\leavevmode\newline
Let $u$ and $v$ be defined as in the statement, then we know that $u - v \approx 0$, \textit{i.e.} we can do again the trick to use measure theory since the zero distribution is also a measure, so for a volume form dvol as in Thm. \ref{thm:geilestheorem} and $f$ a Borel measurable function with compact support we get
\bas
\lim_{\varepsilon \to 0^+} \int_M (v-u)_- ~ f ~ \mathrm{dvol} = 0
\eas
and similarly for $u_-$ by $u \geq 0$. Then we use (where $h$ is another Borel measurable function with compact support)
\bas
-(f-h)_-
&=
\frac{|f-h|-f+h}{2}
\geq
\frac{|f|-|h|-f+h}{2}
=
\frac{|f|-f}{2} + \frac{h-|h|}{2}
=
-f_- + h_-
\eas
to get using $u \geq 0$
\bas
0
\leq
\lim_{\varepsilon \to 0^+} \int_M (-v_-) ~ (\pm f_\pm) ~ \mathrm{dvol}
&=
\lim_{\varepsilon \to 0^+} \int_M (-v_- + u_-) ~ (\pm f_\pm) ~ \mathrm{dvol} \\
&\leq
\lim_{\varepsilon \to 0^+} \int_M (-(v - u)_-) ~ (\pm f_\pm) ~ \mathrm{dvol}
= 0,
\eas
hence $\forall f \in C^\infty_c(M)$ one has $\lim_{\varepsilon \to 0^+}  \int_M v_- ~ f ~ \mathrm{dvol} = 0$ and thence $v \geq 0$ by writing $f = f_+ + f_-$.
\end{proof}

\subsection{NEC}

Using this definition of non-negativity one can now try to define the NEC in sense of the Colombeau algebra. But for this we also need a definition of lightlike vector fields in that setting.

\begin{definition}[Generalized lightlike vector fields]
\leavevmode\newline
The set of lightlike vector fields $\hat{\mathcal{L}}^1_0(M)$ consists of generalized vector fields $X \in \hat{\mathcal{G}}^1_0(M)$ with the following properties
\begin{enumerate}
\item $\iota(g)(X,X) = \iota(g_{ij}) X^i X^j = 0$,
\item $\forall$ charts $(U, \psi)$ with associated derivatives $(\partial_i)_i$ we have the property that
\bas
\forall 0 \leq L \in \hat{\mathcal{G}}(U): ~ \forall i \in {1, \dots, n}: ~
\left(X^i\right)^2 ~ L \geq 0 \text{ in } \hat{\mathcal{G}}(U),
\eas
\end{enumerate}
where $X=\sum_i X^i \partial_i$ locally.
\label{def:lightlike}
\end{definition}

\begin{remark}
\leavevmode\newline
The first condition seems to be natural but beware that this is not generalizing the classical condition "g(X,X) = 0" where we have a continuous metric $g$ and (hence in general continuous) lightlike vector fields $X$. Thence we would look at the term $\iota(g)(\iota(X),\iota(X))$ instead of $g(X,X)$ after embedding $g$ and $X$ into the Colombeau algebra. Multiplications of continuous tensors are not generalized but as explained in Rem. \ref{rem:association} one can generalize them in the level of association, \textit{i.e.} $\iota(g)(\iota(X),\iota(X)) \approx \iota(g(X,X)) = 0$. Thus, one might try to replace the first condition with $\iota(g)(\iota(X),\iota(X)) \approx 0$ to get a  generalized definition. The advantage of our definition is that one can neglect later the terms proportional to $\iota(g)$ similar to the standard situation while this can in general not be done by using the definition with association because association is not stable under multiplications, especially under multiplications with singular terms due to the convergence argument in the definition of association.
\newline The second condition is new and due to the nature of the Colombeau algebra. Allowing any vector field satisfying the first condition would also allow any high order of singularity. Since the definition of non-negativity is a convergence argument it could happen that non-negative elements are not non-negative anymore after multiplying them with the square of the components of such singular vector fields although a square should not disturb non-negativity, naively argued. Take \textit{e.g.} a diagonalized (in some chart) stress energy tensor $T$ proportional to the delta distribution (which is non-negative for the typical used volume forms and orientations), then the contraction $T(X,X)$ consists exactly of terms of the form $\left(X^i\right)^2 T_{ii}$. When $X$ is too singular then the negative part of this might not converge anymore such that one could not say that it is non-negative although one would like to say this. To avoid such cases we need the second condition.
\newline That this definition of lightlike vector fields is suitable for us can be seen in Theorem \ref{thm:NEC}. We will then also see that this set of lightlike vector fields is non-empty for our situation and that we have indeed a generalized situation preserving the classical one in some sense although one has the problem regarding the first condition.
\label{rem:lightlike}
\end{remark}

Finally we are able to define a generalized version of the NEC.

\begin{definition}[Generalized NEC]
\leavevmode\newline
We say that the generalized NEC for the stress-energy tensor $T \in \hat{\mathcal{G}}^0_2(M)$ holds if and only if 
\bas
\forall X \in \hat{\mathcal{L}}^1_0(M):~
T(X,X) = T_{ij} X^i X^j
\geq 0,
\eas
where $X=\sum_i X^i \partial_i$ locally.
\end{definition}

Using this definition one can now start the study of the NEC, so let us go back to the four dimensional case. But for this one would have to check an infinite big set of inequalities. But one can reduce this system to a finite system and that is now the target. For this we need some auxiliary result. To understand the following proposition it is important to understand some features of generalized pseudo-Riemannian geometry like stated in \cite{GenRiem} which will be cited in the proof. But for the statement itself we need already some property stated in \cite{GenRiem}. Let $V \subset M$ be a relatively compact open subset then there is a representative $(g_\varepsilon)_\varepsilon$ of $\iota(g)$ on $V$ such that there is an $\varepsilon_0 > 0$ with the property that $g_\varepsilon$ is a standard smooth pseudo-Riemannian metric (with a well-defined signature) on $V$ for all $\varepsilon < \varepsilon_0$. Our indices are running from 0 to 3 where 0 is the time coordinate and 1, 2 and 3 are denoting the spatial coordinates.

The following result can be skipped if one is not interested into mathematical rigorous approaches; the following proposition basically tells one that one can take the square root of a diagonalized metric (cancelling the sign before) such that the resulting sequences lie in the Colombeau algebra, so one can use them without worries.

\begin{proposition}[Square root of the metric]
\leavevmode\newline
If one has a generalized pseudo-Riemannian metric $\iota(g)$ as in \cite{GenRiem} with index $(-+++)$ (where we keep the notation $\iota(g)$ in context of our situation) which is diagonalized in some coordinate system and chart $(U, \psi)$ then one can well-define the square root $\sqrt{c_i ~ \iota(g)_{ii}}$ (where $c_0 := -1$, $c_\nu := 1$ for $\nu \in \{1,2,3\}$) on any fixed relatively compact open subset $V \subset U$ by
\ba
\sqrt{c_i ~ \iota(g)_{ii}} := \sqrt{c_i ~ (g_{ii})_\varepsilon} \in \hat{\mathcal{G}}(V),
\ea
where $(g_{ii})_\varepsilon$ is defined as above and fixed. In sense of Colombeau, i.e. in sense of the quotient structure we also then clearly get
\ba
\left(\sqrt{c_i ~ \iota(g)_{ii}}\right)^2
&=
c_i ~ \iota(g)_{ii}.
\label{eq:klar}
\ea
\label{prop:squarroot}
\end{proposition}

\begin{proof}
\leavevmode\newline
We have to check if we can well-define the square root of the diagonal elements and if they represent elements of the Colombeau algebra, \textit{i.e.} if the sequence associated to it is moderate (see Def. \ref{def:Colombeau}). So let us look at $\sqrt{c_i~ \iota(g_{ii})}$ where $c_0 = -1$ and $c_\nu=1$ for $\nu \in \{1,2,3\}$.
\newline That the square root can be well-defined for a fixed (and small) $\varepsilon$ is not that difficult in our case. We only have to define the square root in relatively compact open subsets $V \subset U$; as explained before the proposition we know that on $V$ there is an $\varepsilon_0$ such that for all $\varepsilon < \varepsilon_0$ there is a representative sequence $g_\varepsilon$ (so a sequence in the equivalence class) of $\iota(g)$ which describes a classical pseudo-Riemannian metric with index $(-+++)$. Let us fix this sequence on $V$ and by taking this sequence it is clear that $\sqrt{c_i~ \iota(g_{ii})} := \left(\sqrt{c_i ~ (g_{ii})_\varepsilon}\right)_\varepsilon$ is well-defined on $V$ for all $\varepsilon < \varepsilon_0$. Assuming that this defines an element in the Colombeau algebra, it is also clear that one then gets $\left(\sqrt{c_i~ \iota(g_{ii})}\right)^2 = c_i~ \iota(g_{ii})$ in sense of the Colombeau algebra since $[(g_\varepsilon)_\varepsilon] = \iota(g)$.
\newline Now we have to discuss if the sequence associated to the square root is moderate such that it well-defines an element in the Colombeau algebra $\hat{\mathcal{G}}(V)$. Since the metric is diagonal in this frame and has an inverse as discussed in \cite{GenRiem} (Section 2; it is simply $(\iota(g))^{-1} = [((g_\varepsilon)^{-1})_\varepsilon]$) we can conclude that every diagonal element is invertible and, thus, also shown in \cite{GenRiem} we get for all components
\ba
\forall K \subset V \text{ compact}: ~
\exists \tilde{\varepsilon}_0 > 0: ~ \exists q \in \mathds{N}: ~ \forall \varepsilon < \min\{\varepsilon_0, \tilde{\varepsilon}_0\}: ~\inf_{p \in K} \left| (g_{ii})_\varepsilon(p) \right| &\geq \varepsilon^q,
\label{eq:genmetric}
\ea
since $g_\varepsilon$ is in the equivalence class of $\iota(g)$. As argued in \cite{GenRiem}, Section 2, this is an equivalent condition for invertibility. To check that the square root sequence is moderate we have to prove that for all compact subsets $K$ of $V$ (see Def. \ref{def:Colombeau})
\bas
\forall \xi \in \mathds{N}_0^n: ~ \exists N(\xi) \in \mathds{N}: ~
\sup_{p \in K} \left|\mathrm{D}^\xi\left(\sqrt{c_i ~ \iota(g_{ii})}\right)(p)\right| = \mathcal{O}\left(\varepsilon^{-N(\xi)}\right) ~ (\varepsilon \to 0^+),
\eas
where we replaced the Lie derivatives with the typical derivative since we look at the scalar case (with global coordinates). For $g_\varepsilon$ we already know that it is moderate, \textit{i.e.} for all $K$
\ba
\forall \xi \in \mathds{N}_0^n: ~ \exists m(\xi) \in \mathds{N}: ~
\sup_{p \in K} \left|\mathrm{D}^\xi\left((g_{ii})_\varepsilon\right)(p)\right| = \mathcal{O}\left(\varepsilon^{-m(\xi)}\right) ~ (\varepsilon \to 0^+).
\label{eq:metricmoderate}
\ea
Hence, for $\xi = 0$ it is then clear that (using basic properties of the square root with respect of finding the supremum \textit{etc.} and $\varepsilon < \varepsilon_0$)
\bas
&&\overline{\lim_{\varepsilon \to 0^+}}\left( \sup_{p \in K} \left|\sqrt{c_i ~ (g_{ii})_\varepsilon}(p)\right| \varepsilon^{m(0)/2} \right)
&=
\sqrt{\overline{\lim_{\varepsilon \to 0^+}} \sup_{p \in K} ~ c_i ~ (g_{ii})_\varepsilon ~ \varepsilon^{m(0)}}
< \infty \\
&\Leftrightarrow&
\sup_{p \in K} \left|\sqrt{c_i ~ (g_{ii})_\varepsilon}(p)\right|
&=
\mathcal{O}\left( \varepsilon^{-m(0)/2} \right) ~ (\varepsilon \to 0^+).
\eas
To look at the derivatives first observe that the derivatives can be taken without the problem of a singularity since $g_\varepsilon$ is a classical pseudo-Riemannian metric; so the sequence associated to $\sqrt{c_i ~ \iota(g_{ii})}$ is a sequence of smooth functions (which is important of course). $\mathrm{D}^\xi\left(\sqrt{c_i ~ \iota(g_{ii})}\right)$ (for $\xi \neq 0$) is basically a finite sum of elements of the form
\bas
\frac{d}{\left(\sqrt{c_i ~ (g_{ii})_\varepsilon}\right)^t} ~ \prod_{\text{finitely many } s \in \mathds{N}^n_0} \mathrm{D}^s((g_{ii})_\varepsilon),
\eas
where $t \in \mathds{N}$ and $d\in\mathds{R}$. Then for $\varepsilon < \min\{\varepsilon_0, \tilde{\varepsilon}_0\}$ using Ineq. \eqref{eq:genmetric} and Eq. \eqref{eq:metricmoderate} and for $r \in \mathds{N}$ we get
\bas
&\overline{\lim_{\varepsilon \to 0^+}}\left( \sup_{p \in K} \left|
\frac{d}{\left(\sqrt{c_i ~ (g_{ii})_\varepsilon}\right)^t} ~ \prod_{\text{finitely many } s \in \mathds{N}^n_0} \mathrm{D}^s((g_{ii})_\varepsilon)
\right| ~ \varepsilon^r \right) \\
&\leq 
\overline{\lim_{\varepsilon \to 0^+}}\left(
\frac{\varepsilon^r ~|d|}{\left(\sqrt{\inf_{p \in K}|(g_{ii})_\varepsilon|}\right)^t} ~ \prod_{\text{finitely many } s \in \mathds{N}^n_0} \sup_{p \in K} \left|\mathrm{D}^s((g_{ii})_\varepsilon)\right|
\right) \\
&\leq
|d|~\overline{\lim_{\varepsilon \to 0^+}}\left(\varepsilon^{r-qt/2-\sum_{\text{finitely many } s \in \mathds{N}^n_0} m(s)}\right) ~ \prod_{\text{finitely many } s \in \mathds{N}^n_0} \underbrace{\overline{\lim_{\varepsilon \to 0^+}}\left( \sup_{p \in K} \left|\mathrm{D}^s((g_{ii})_\varepsilon)\right| \varepsilon^{m(s)} \right)}_{< \infty \text{ by Eq. \eqref{eq:metricmoderate}}} \\
&<
\infty \text{ for } r \geq \frac{qt}{2} + \sum_{\text{finitely many } s \in \mathds{N}^n_0} m(s).
\eas
Since $\mathrm{D}^\xi\left(\sqrt{c_i ~ \iota(g_{ii})}\right)$ is a finite sum of such terms and since $K$ and $\xi$ were arbitrary, one gets that $\sqrt{c_i ~ \iota(g_{ii})}$ is a moderate sequence. Thence, it defines an element of $\hat{\mathcal{G}}(V)$ which we will also denote by $\sqrt{c_i ~ \iota(g_{ii})}$. Eq. \eqref{eq:klar} is clear.
\end{proof}

Using this we can now simplify in some situation the generalized NEC for a 4-manifold with an index $(-+++)$.

\begin{theorem}[Special form of the NEC for Hawking-Ellis type I]
\leavevmode\newline
Let $(U, \psi)$ be a chart on which both $T$ and $\iota(g)$ are diagonal where $g$ is a continuous pseudo-Riemannian metric with index $(-+++)$. Then we have in this coordinate system
\ba
\forall X \in \hat{\mathcal{L}}^1_0(U):~
T(X,X)
\geq 0
&~\Leftrightarrow~
\forall \nu \in \{1,2,3\}: ~
0
\leq
T_{00} ~ \iota(g_{\nu\nu}) - T_{\nu\nu} ~ \iota(g_{00}) \text{ on U}.
\ea
\label{thm:NEC}
\end{theorem}

\begin{remark}
\leavevmode
\newline It will be clear by the proof that this result can also be used for any other element in $\hat{\mathcal{G}}^0_2(M)$ instead of $T$.
\newline
To see the connection to the classical case look at this inequality as if it were a classical one, \textit{i.e.}
\bas
0 \leq T_{00} ~ g_{\nu\nu} - T_{\nu\nu} ~ g_{00}
\eas
and multiply this with $0 < - g^{00} ~ g^{\nu\nu}$ such that one gets
\bas
0 \leq - T_{~0}^{0} + T_{~\nu}^\nu.
\eas
This is a typical used inequality in the literature where one has a similar theorem like Thm. \ref{thm:NEC} and it is one of many reasons why $-T_{~0}^{0}$ is the energy density and $T_{~\nu}^\nu$ a principal pressure (up to a sign); see \textit{e.g.} \cite{Visser}, Chapter 12, or \cite{energiedichte}, Section 3, to see the physical interpretation of the components. In the same manner one could also prove this version of the inequality in the Colombeau algebra (since multiplications with the metric are nice enough which one can observe in the proof below). We do not want to use this version because we want to avoid multiplications with the inverse of the metric which is more tedious in the Colombeau algebra. After the proof of this theorem we will also argue that we have a generalized situation preserving the classical one in sense of association.
\end{remark}

\begin{proof}
\leavevmode\newline
So let $X$ be a generalized lightlike vector field in $U$ (see Def. \ref{def:lightlike}) and we have a local coordinate system $(\partial_0, \dots, \partial_3)$ in $U$ such that $T=(T_{ij})_{i,j=0}^3$ is diagonal in this coordinate frame; the metric has the index $(- + + +)$ as usual and has also a diagonal form in these local coordinates.
\newline "$\underline{\Rightarrow}$:"
\newline Let $w \in \Omega_c^n(U)$ be arbitrary but fixed and take any relatively compact open subset $U \supset V \supset \mathrm{supp}(w)$ such that we can define $\sqrt{c_i ~ g_{ii}}$ on $V$ by Prop. \ref{prop:squarroot} (recall there the definition of $g_\varepsilon$). Take in the given coordinate frame the vector field (recall that $\nu\in\{1,2,3\}$)
\bas
X
&:=
\sqrt{\iota(g_{\nu\nu})} ~ \partial_0 + \sqrt{-\iota(g_{00})} ~ \partial_\nu
\in \hat{\mathcal{G}}^1_0(V), 
\eas
such that (using that everything is now diagonal)
\bas
\iota^0_2(g)(X,X) 
&=
\iota(g_{00}) ~ \iota(g_{\nu\nu})
- \iota(g_{\nu\nu}) ~ \iota(g_{00})
= 0.
\eas
Since $\iota(g)$ is given as a convolution of a continuous metric we know that $\iota(g)$ converges uniformly to $g$ on compact subsets (a basic property of the convolution). By this and since $\iota(g)$ is a sequence of bounded tensors on a compact subset it follows that this sequence (and so also $g_\varepsilon$ by Def. \ref{def:Colombeau}, especially by the definition of negligible tensors) is uniformly bounded by some constant $Z$ on this compact subset (\textit{e.g.} on $\overline{V}$ which is compact by definition and so also on $V \subset \overline{V}$). For all $w \in \Omega_c^n(V)$ ($c_0 := -1$, $c_\nu := 1$) and for all $0 \leq L \in \hat{\mathcal{G}}(V)$ one wants to show for each $i \in \{0,\dots,3\}$ (in the following calculation we mean with "$>0$" of course the classical one)\footnote{$\stackrel{!}{=} 0$ just means that we want that this is zero and that we want to check if it is zero. Similar with $\stackrel{!}{\leq}$ \textit{etc}..}
\ba
\lim_{\varepsilon \to 0^+} \int_V \biggl(\underbrace{\left(\sqrt{c_i ~ \iota(g_{ii})}\right)^2}_{\makebox[0mm][c]{\footnotesize$= c_i ~ (g_{ii})_\varepsilon > 0$}} ~ L \biggr)_- ~ w
&=
\lim_{\varepsilon \to 0^+} \int_V c_i ~ (g_{ii})_\varepsilon ~ (L)_- ~ w
\stackrel{!}{=} 0.
\label{convergenceofvectorfield}
\ea
This can be shown by noticing that $(L)_-$ represents the zero distribution \textit{i.e.} the trivial measure (see the proof of Thm \ref{thm:geilestheorem} for the techniques) such that one can extend the domain of definition to Borel measurable $n$-forms and so we can again transform this problem into measure theory. So fixing a volume form dvol with properties as stated in Thm. \ref{thm:geilestheorem} one can write again as usual $w =f ~ \mathrm{dvol}$ (where $f$ is a Borel measurable function with compact support) such that also $\lim_{\varepsilon \to 0^+} \int_V (L)_- ~ (f)_\pm ~ \mathrm{dvol} = 0$ and so
\bas
0\leq\int_V \underbrace{c_i ~ (g_{ii})_\varepsilon}_{>0} ~\underbrace{\left( \mp (L)_- ~ (f)_\pm\right)}_{\geq 0} ~ \mathrm{dvol}
&\leq
\mp Z \int_V (L)_- ~ (f)_\pm ~ \mathrm{dvol}
\stackrel{\varepsilon \to 0^+}{\longrightarrow} 0,
\eas
such that clearly $\int_V c_i ~ (g_{ii})_\varepsilon ~ (L)_- ~ w \stackrel{\varepsilon \to 0^+}{\longrightarrow} 0$ by splitting $w$ or rather $f$ into the positive and negative part and therefore Condition \eqref{convergenceofvectorfield} holds. 
\newline Since a coordinate change of $X$ in open subsets of V is done by multiplication with smooth functions, components of the Jacobi matrix (do not transform the metric under the square root, they are fixed with respect to $V$ which is well-defined since $V$ has a global coordinate system given by $U$), one can do exactly the same steps as before to show that this stability criteria for the non-negativity of $L$ under a multiplication with squares of $X$ holds in any coordinate system and so $X \in \hat{\mathcal{L}}^1_0(V)$. By extending $X$ with zero one gets easily a lightlike vector field on whole of $M$ (and so on $U$) which also shows that $\hat{\mathcal{L}}^1_0(M)$ is non-empty when one can diagonalize the metric in some open subset which is clearly given in the case of our taken thin shell wormhole.
\newline Then by assumption (also recall that the definition of non-negativity is independent of the taken representative sequence such that we can replace $(g_{ii})_\varepsilon$ by $\iota(g)$; see Rem. \ref{rem:nonneg})
\bas
0
&=
\lim_{\varepsilon \to 0^+}\int_M (T(X,X))_- ~ w
=
\lim_{\varepsilon \to 0^+}\int_U (T_{00} ~ \iota(g_{\nu\nu}) - T_{\nu\nu} ~ \iota(g_{00}))_- ~ w 
\eas
and since this can be done for any $w$ (redefining $X$ since -rigorously spoken- one has to change the definition of the square root)
\bas
0
&\leq
T_{00} ~ \iota(g_{\nu\nu}) - T_{\nu\nu} ~ \iota(g_{00}) \text{ on U}.
\eas
Since $\nu$ was arbitrary this holds for all $\nu \in \{1,2,3\}$ in this fixed coordinate frame.
\newline "$\underline{\Leftarrow}$:"
\newline Let us keep the same notation as before and now we know that
\bas
\forall \nu \in \{1,2,3\}: ~
0
&\leq
T_{00} ~ \iota(g_{\nu\nu}) - T_{\nu\nu} ~ \iota(g_{00}) \text{ on U}.
\eas
Now let us take again a lightlike vector field $X = X^i ~ \partial_i \in \hat{\mathcal{L}}^1_0(U)$ which is now arbitrary but with the following property on $U$ (again using that everything is diagonal and that $\iota(g)$ is a metric with an inverse)
\bas
&&0
&=
\iota^0_2(g)(X,X)
= \iota(g_{00}) \left(X^0\right)^2 + \iota(g_{11}) \left(X^1\right)^2 + \iota(g_{22}) \left(X^2\right)^2 + \iota(g_{33}) \left(X^3\right)^2 \\
&\Leftrightarrow&
\left(X^0\right)^2
&=
- \iota(g_{00})^{-1} \left( \iota(g_{11}) \left(X^1\right)^2 + \iota(g_{22}) \left(X^2\right)^2 + \iota(g_{33}) \left(X^3\right)^2 \right).
\eas
Hence we have for all $w \in \Omega_c^n(U)$ 
\ba
&\int_M (T(X,X))_- ~ w \nonumber \\
&=
\int_{\mathrm{supp}(w)} \left( - T_{00} ~ \iota(g_{00})^{-1} \left( \iota(g_{11}) \left(X^1\right)^2 + \iota(g_{22}) \left(X^2\right)^2 + \iota(g_{33}) \left(X^3\right)^2 \right)
+ T_{11} \left(X^1\right)^2 \right. \nonumber \\
&\qquad\qquad\quad
\left.+ T_{22} \left(X^2\right)^2
+ T_{33} \left(X^3\right)^2 \right)_- ~ w 
\label{eq:stressenergyintegral}
\ea
We also have (using again that $- \left( (g_{00})_\varepsilon \right)^{-1}$ is positive on relatively compact open subsets like we did it in the proof of Prop. \ref{prop:squarroot})
\bas
&I(w) \\
&:=
\int_{\mathrm{supp}(w)} \left( -T_{00} \left( (g_{00})_\varepsilon \right)^{-1} \left( \iota(g_{11}) \left(X^1\right)^2 + \iota(g_{22}) \left(X^2\right)^2 + \iota(g_{33}) \left(X^3\right)^2 \right) + T_{11} \left(X^1\right)^2  \right. \\
&\qquad\qquad\quad~
\left.
+ T_{22} \left(X^2\right)^2
+ T_{33} \left(X^3\right)^2 \right)_- ~ w \\
&=
\int_{\mathrm{supp}(w)} \left(-\left((g_{00})_\varepsilon\right)^{-1}\right) \left( 
\left( X^1 \right)^2 \left( T_{00} ~ \iota(g_{11}) - T_{11} ~ (g_{00})_\varepsilon \right)
+ \left( X^2 \right)^2 \left( T_{00} ~ \iota(g_{22}) - T_{22} ~ (g_{00})_\varepsilon \right) \right. \\
&\qquad\qquad\qquad\qquad\qquad\quad~+
\left. \left( X^3 \right)^2 \left( T_{00} ~ \iota(g_{33}) - T_{33} ~ (g_{00})_\varepsilon \right) \right)_- ~ w.
\eas
One can rewrite the integral of the stress-energy tensor using again $-(f + h)_- \leq -f_- -h_-$ for Borel measurable functions $f$ and $h$ as explained in Rem. \ref{rem:nonneg} (by using also the same trick as before, \textit{i.e.} $w = f_+ ~ \mathrm{dvol} + f_- ~ \mathrm{dvol}$ where $f$ is a measurable function and dvol is the fixed volume form as used before; this is clearly valid here since for a fixed $\varepsilon$ the integral describes a measure)
\bas
0 
&\leq 
I(\mp f_\pm ~ \mathrm{dvol}) \\
&\leq
-\int_{\mathrm{supp}(w)} \left(- \left((g_{00})_\varepsilon\right)^{-1}\right)
\left( \left( X^1 \right)^2 \left( T_{00} ~ \iota(g_{11})
- T_{11} ~ (g_{00})_\varepsilon \right) \right)_- ~ (\pm f_\pm) ~ \mathrm{dvol} \\
&\quad-
\int_{\mathrm{supp}(w)} \left(- \left((g_{00})_\varepsilon\right)^{-1}\right)
\left( \left( X^2 \right)^2 \left( T_{00} ~ \iota(g_{22})
- T_{22} ~ (g_{00})_\varepsilon \right) \right)_- ~ (\pm f_\pm) ~ \mathrm{dvol} \\
&\quad-
\int_{\mathrm{supp}(w)} \left(- \left((g_{00})_\varepsilon\right)^{-1}\right)
\left( \left( X^3 \right)^2 \left( T_{00} ~ \iota(g_{33})
- T_{33} ~ (g_{00})_\varepsilon \right) \right)_- ~ (\pm f_\pm) ~ \mathrm{dvol}.
\eas
By assumption and definition of generalized lightlike vector fields we know that \linebreak $\left( X^\nu \right)^2 \left( T_{00} ~ \iota(g_{\nu\nu}) - T_{\nu\nu} ~ (g_{00})_\varepsilon \right) \geq 0$ $\forall \nu \in \{1,2,3\}$ in sense of Def. \ref{def:nonnegCol}. One can now prove that $I(w) \stackrel{\varepsilon \to 0^+}{\longrightarrow}0$ $\forall w \in \Omega_c^n(U)$ by applying the same argumentation as one did in proving of Eq. \eqref{convergenceofvectorfield} to show that all three integrals are vanishing since $\mathrm{supp}(w)$ is compact and $g_\varepsilon$ describes a diagonal metric (so the diagonal elements do not reach zero and hence the inverses are also bounded for a fixed $\varepsilon$) and since in \cite{GenRiem}, Prop. 3.8, it is proven that $\iota(g)^{-1}$ converges uniformly to $g^{-1}$ on compact subsets (based on the fact that $\iota(g)$ converges uniformly to $g$ on compact subsets). By $I(w) = I(f_+ \mathrm{dvol}) + I(f_- \mathrm{dvol})$ one clearly gets $I(w) \stackrel{\varepsilon \to 0^+}{\longrightarrow}0$ and hence also $\int_M (T(X,X))_- w \stackrel{\varepsilon \to 0^+}{\longrightarrow}0$ by replacing $\iota(g_{00})^{-1}$ with $((g_{00})_\varepsilon)^{-1}$ in Eq. \eqref{eq:stressenergyintegral} and using that non-negativity is independent of the taken representative sequence (see Rem. \ref{rem:nonneg}).
\end{proof}

This theorem (especially the techniques in its proof) also shows that we have indeed a generalized situation compared to the classical distributional case of a thin shell wormhole as defined in Def. \ref{def:thinshellwormhole} using the standard Einstein field equations (so taking $F(x) = x + 2 \Lambda$; for simplicity we take $\Lambda=0$). In the classical case the stress-energy tensor is proportional to the delta distribution, as already argued. This means that the stress-energy tensor is a distribution of zero order since it does not need the information about derivatives of the bump forms (recall the definition of zero order distributions in the proof of Thm. \ref{thm:geilestheorem}) and thus one can well-define the multiplication of the stress-energy tensor with continuous functions; if you do not know this property one can also argue with Thm. \ref{thm:deltcompetaexp}. Personally I assume that for all distributions $L$ of zero order 
\ba
\forall f \in C(M): ~ \iota(f) ~ \iota(L) \approx \iota(f L)
\label{eq:wichtig}
\ea
holds. For smooth functions $f$ this is shown for any distribution at the very end of \cite{GrundlagenI}; one should be able to prove \eqref{eq:wichtig} in exactly the same manner but one has to restrict oneself onto distributions $L$ of zero order to well-define $f L$ for continuous $f$.

To compare the classical theory and the one given by the Colombeau algebra one calculates the stress-energy tensor in both theories; the embedding $\iota$ is defined by mollifying tensor components w.r.t. to the coordinates $(t_+, \tilde{\eta}, \vartheta, \varphi)$ as defined in Def. \ref{def:thinshellwormhole}. $\tilde{T}$ denotes the stress-energy tensor of the classical theory and $T$ the one given in the Colombeau algebra which is calculated by embedding the metric into the Colombeau algebra and by calculating all curvature terms with the embedded metric, \textit{i.e.} $T := \mathrm{Ric}[\iota(g)] - \frac{1}{2} R[\iota(g)] \iota(g)$ where $\cdot[\iota(g)]$ means the curvature term given in the Colombeau algebra (see also again \cite{GenRiem} that the formulas are exactly the same but now with metric sequences). In \cite{Sobolev} it is shown that the curvature quantities of both theories are associated to each other (especially $\iota(R) \approx R[\iota(g)]$) and then one can clearly show that $T \approx \iota(\tilde{T})$ by using again that the multiplication with the metric does not disturb the association, \textit{i.e.} $\iota(R) \iota(g) \approx R[\iota(g)] \iota(g)$ similar proven as the Eq. \eqref{convergenceofvectorfield}, so using uniform convergence again on the elements of the diagonal metric and using that $(\iota(R) - R[\iota(g)])_\pm \approx 0$ since $\iota(R) - R[\iota(g)] \approx 0$, then one gets $((\iota(R) - R[\iota(g)]) ~ \iota(g)_{ii})_\pm \approx 0$ for all $i$ which shows $\iota(R) \iota(g) \approx R[\iota(g)] \iota(g)$. Then use Eq. \eqref{eq:wichtig} to get $T \approx \iota(\tilde{T})$.

Therefore one gets $T \approx \iota(\tilde{T})$ and thus also $T_{ij} \approx \iota(\tilde{T})_{ij} = \iota(\tilde{T}_{ij})$ for all $i,j \in \{t_+, \tilde{\eta}, \vartheta, \varphi\}$ by the definition of our taken $\iota$. So one gets again by using the stability of association under the multiplication with components of the metric and by Eq. \eqref{eq:wichtig} that $T_{00} ~ \iota(g_{\nu\nu}) - T_{\nu\nu} ~ \iota(g_{00}) \approx \iota(\tilde{T}_{00} ~ g_{\nu\nu} - \tilde{T}_{\nu\nu} ~ g_{00})$ for any $\nu \in \{1,2,3\}$ such that one could use again Cor. \ref{cor:NECgleich} on the study of non-negativity of these three inequalities. Thus, one sees that the study of the NEC of our taken thin shell wormhole is indeed generalized in sense of association (recall also that the non-negativity in sense of Colombeau generalizes the classical distributional sense of non-negativity). Moreover, by Rem. \ref{rem:nonneg} we know that the sense of non-negativity in the Colombeau algebra used on embedded distributions is an equivalent definition of the sense of non-negativity of distributions (w.r.t. a suitable fixed volume form) such that one has
\bas
0 \leq T_{00} ~ \iota(g_{\nu\nu}) - T_{\nu\nu} ~ \iota(g_{00}) 
~\Leftrightarrow~
0 \leq \iota(\tilde{T}_{00} ~ g_{\nu\nu} - \tilde{T}_{\nu\nu} ~ g_{00})
~\Leftrightarrow~
0 \leq \tilde{T}_{00} ~ g_{\nu\nu} - \tilde{T}_{\nu\nu} ~ g_{00},
\eas
\textit{i.e.} using the Colombeau algebra on the classical situation gives an equivalent description. Finally one can observe in the proof of the previous theorem that one gets back the classical three inequalities especially by the definition of $\hat{\mathcal{L}}^1_0(M)$; this and the previous discussion might be seen as a confirmation for taking such lightlike vector fields although of the problem discussed in Rem. \ref{rem:lightlike}.

\section{Conjecture}

Now we have everything to start the study about the NEC. So let us go back to the spacetime $M$ defined as in Def. \ref{def:thinshellwormhole}. Then we calculate the stress-energy tensor as defined in Def. \ref{def:F(R)} by using the Colombeau algebra. We will use quadratic $F(x) = 2 \Lambda + x + \alpha_2 x^2$ such that we get in the Colombeau algebra (recall the notation of $\cdot[\iota(g)]$ as above)
\ba
\kappa T^{\mathrm{NEC}}
=&~
(1 + 2 \alpha_2 ~ R[\iota(g)]) ~ \mathrm{Ric}[\iota(g)]
- 2 \alpha_2 \cdot \nabla^2(R[\iota(g)]),
\label{eq:stressenergyrigorous}
\ea
where $T^{\mathrm{NEC}}$ means the part of the stress-energy tensor needed for the NEC, \textit{i.e.} all parts proportional to $\iota(g)$ are put to zero by the definition of $\hat{\mathcal{L}}^1_0(M)$, see Def. \ref{def:lightlike}, since we will look at $T(X,X)$ for $X \in \hat{\mathcal{L}}^1_0(M)$. To calculate this we have to use the Colombeau algebra now and we will use for simplicity the Schwarzschild spacetime on both sides of the wormhole (with the same mass $\tilde{M}$), \textit{i.e.} $A^\pm(r_\pm) := 1 - \frac{2 \tilde{M}}{r_\pm}$ and $B^\pm := \left( A^\pm \right)^{-1}$, $a > 2 \tilde{M}$ (recall the beginning of section 2; $a$ describes the radius of the wormhole throat which shall be big enough to prevent event horizons which would lie at $2\tilde{M}$). As one can see in Def. \ref{def:thinshellwormhole} the kink of the metric is only along the proper radial distance $\tilde{\eta}$ and the whole universe itself is also spherical symmetric. Thus, we only need to approximate along the proper radial distance, so one approximates the function only along $\tilde{\eta}$ like it is also done in \cite{Schwarzschild} where the Schwarzschild metric is embedded into the Colombeau algebra. For this take an admissible mollifier (see Lemma \ref{lem:admismoll}) $\rho_\varepsilon \in C^\infty_c(\mathds{R})$ in one dimension. Since $\eta$ is given by $\tilde{\eta}$ in coordinates one gets \textit{e.g.} $\iota(\Theta \circ \eta) = (\Theta * \rho_\varepsilon) \circ \eta$. As discussed previously the Colombeau algebra is generalizing the Lie derivative which means that the Lie deriviative commutes with $\iota$. One can then easily check by calculation that (for $X:= \mathrm{grad}(\eta)/||\mathrm{grad}(\eta)||^2 \stackrel{\text{here}}{=} \partial_{\tilde{\eta}}$; recall also Def. \ref{def:compwitheta}) the following holds
\bas
\iota(\delta \circ \eta)
&=
\iota\left(\mathcal{L}_X(\Theta \circ \eta)\right)
= \mathcal{L}_X(\iota(\Theta \circ \eta))
= \mathcal{L}_X((\Theta * \rho_\varepsilon) \circ \eta)
= \rho_\varepsilon \circ \eta
\eas
and by this and by construction of the Colombeau algebra it is also clear that for all $w \in \Omega_c^n(M)$
\ba
(\delta \circ \eta)(w)
&=
\lim_{\varepsilon \to 0} \int_M \iota(\delta \circ \eta) ~ w
=
\lim_{\varepsilon \to 0} \int_M (\rho_\varepsilon \circ \eta) ~ w
\ea
holds. This shows that Def. \ref{def:compwitheta} is equivalent to the standard way to define $\delta \circ \eta$ at least for our situation, but this argument could be easily extended to the general case. All the other arising non-smooth functions can be embedded in a similar manner.

Now one calculates the stress-energy tensor; when one is doing this one will realize some big disadvantage of the Colombeau algebra: It is much harder to calculate and it is not clear if one can study the NEC then without use of numerics. Thus, we want to simplify the following calculations by analysing the additional information of the Colombeau algebra compared to the distributional framework and using suitable approximations. The wormhole has a kink in the metric in the distributional framework but when one embeds the wormhole into the Colombeau algebra one gets a sequence of smooth metrics approximating the initial metric. As in the classical case of standard Einstein field equations one would like that the stress-energy tensor will have a support in the thin shell (for $\varepsilon \to 0$) since one has Schwarzschild spacetimes (a vacuum solution of the standard Einstein field equations) off the thin shell. By inserting the curvature terms into Eq. \eqref{eq:F(S)} without using the Colombeau algebra one would naively argue that the support has to remain on the thin shell also in the $F(R)$-gravity because the support can not increase after multiplications and derivatives. This stress-energy is given by a sequence of stress-energy tensors given by the previously mentioned sequence of smooth metrics. Since this sequence of smooth metrics approximates the initial metric with a kink (where the matter lies) one can naively expect that for small and fixed $\varepsilon$ the matter lies on a spherical shell around the thin shell, \textit{i.e.} the support of the stress-energy tensor for small but fixed $\varepsilon$ is basically a spherical shell around the thin shell. This means that one has additional information about matter along the proper radial distance, the so-called \textit{microstructure}.

In the classical case using the standard Einstein field equations this additional information vanishes for $\varepsilon \to 0$ (also seen by $f(x) ~ \delta(x) = f(0) ~ \delta(x)$) which is the reason why one calls this microstructure since this is basically an infinitesimal information about matter; and so only the exotic matter survives in the classical situation, \textit{exotic} because this remaining matter violates the NEC. But now we have a highly non-linear situation with multiplications with arbitrary high singularities and thus it could happen that the microstructure (multiplied with such a singularity) does not vanish anymore as $\varepsilon \to 0$; see \textit{e.g.} \cite{VickerAnfang}, Section 6 and 7, or \cite{Multiplikation} for another physical examples about an interpretation of this arising microstructure in non-linear theories. In the conclusion is also an extended discussion about microstructure referring to another examples. This microstructure is basically the additional new physical information one gets using the Colombeau algebra. We will now try to make the microstructure as small as possible to make the "step of generalization" (the difference of this framework to the distributional framework) as small as possible since allowing any big step of generalization would always lead to a theory in which a desired statement holds. Then it is also easier to compare the results with the classical situation.

To do this we will do any calculation in the distributional framework which can be done in the distributional framework. Hence, we will calculate every curvature term in the distributional framework and then we embed these terms independently into the Colombeau algebra instead of embedding the metric and calculating the curvature terms in the Colombeau algebra. When one is doing this one gets then for the stress-energy tensor instead the following formula (compare it with Eq. \eqref{eq:stressenergyrigorous})
\ba
\kappa T^{\mathrm{NEC}}
=&~
(1 + 2 \alpha_2 ~ \iota(R)) ~ \iota(\mathrm{Ric})
- 2 \alpha_2 \cdot \nabla^2(\iota(R)).
\ea
After a lot of straightforward calculations one gets the following results $\big($recall Prop. \ref{prop:deltcompetaexp} and Thm. \ref{thm:NEC} and its remark; we also assume that Eq. \eqref{eq:wichtig} holds and we use Eq. \eqref{eq:wichtig} on the terms which do not contain $\alpha_2$ to get back the classical results. Thus, one gets an $\approx$ and that is allowed for the study of the NEC by Cor. \ref{cor:NECgleich}. The prime denotes the derivative along $\tilde{\eta}$, \textit{i.e.} (for $k \in \mathds{N}$) $\delta^{(k)} \circ \eta := \mathcal{L}^k_{\partial_{\tilde{\eta}}}(\delta \circ \eta)$ and it easily follows that $\iota(\delta^{(k)} \circ \eta) = \rho_\varepsilon^{(k)} \circ \eta$$\big)$
\ba
\kappa \sigma
:=&~ \kappa T^{\mathrm{NEC}}_{t_+t_+} - \kappa T^{\mathrm{NEC}}_{\tilde{\eta}\tilde{\eta}} \iota(g_{t_+t_+}) \nonumber \\
\approx&~
- \frac{4}{a} ~ \left(1-\frac{2 \tilde{M}}{a}\right)^{3/2} ~ \iota(\delta \circ \eta)
- 2 \alpha_2 ~ \alpha ~ \left( \left(\mathfrak{\Gamma}^-\right)^{\tilde{\eta}}_{~t_+t_+} + \left(\mathfrak{\Gamma}^+\right)^{\tilde{\eta}}_{~t_+t_+} \right) ~ \iota(\delta' \circ \eta) \nonumber \\
&
- 2\iota(g_{t_+t_+}) ~ \alpha_2 ~ \alpha ~ \iota(\delta'' \circ \eta)
- \left( \iota(g_{t_+t_+}) ~ \alpha^2 + 2 \alpha ~ \beta \right) ~ \alpha_2 ~ (\iota(\delta \circ \eta))^2, \label{ineq1} \\
\kappa (\sigma - \nu_\vartheta)
:=& ~\kappa T^{\mathrm{NEC}}_{t_+t_+} \iota(g_{\vartheta\vartheta}) - \kappa T^{\mathrm{NEC}}_{\vartheta\vartheta} \iota(g_{t_+t_+}) \nonumber \\
\approx&~
\left( 6 \tilde{M} -2 a \right) \sqrt{1 - \frac{2 \tilde{M}}{a}} ~ \iota(\delta \circ \eta)
- 2 \alpha_2 \alpha \left( \beta ~ \iota(g_{\vartheta\vartheta}) + \gamma ~ \iota(g_{t_+t_+}) \right) (\iota(\delta \circ \eta))^2 \nonumber \\
&
+ 2 \alpha_2 ~ \alpha
\left(
\iota(g_{t_+t_+}) \left( \left(\mathfrak{\Gamma}^-\right)^{\tilde{\eta}}_{~\vartheta\vartheta} + \left(\mathfrak{\Gamma}^+\right)^{\tilde{\eta}}_{~\vartheta\vartheta} \right)
- \iota(g_{\vartheta\vartheta}) \left( \left(\mathfrak{\Gamma}^-\right)^{\tilde{\eta}}_{~t_+t_+} + \left(\mathfrak{\Gamma}^+\right)^{\tilde{\eta}}_{~t_+t_+} \right)
\right) \nonumber \\
&\quad~\cdot
\iota(\delta' \circ \eta),
\label{ineq2}
\ea
where
\ba
\alpha
&:=
\sqrt{\frac{a}{a-2\tilde{M}}} ~ \frac{4 \tilde{M}}{a^2}
+ \frac{8}{a} ~ \sqrt{1-\frac{2 \tilde{M}}{a}},
&\beta
&:= \sqrt{1-\frac{2 \tilde{M}}{a}} ~ \frac{2 \tilde{M}}{a^2},
&\gamma
&:= 2 a \sqrt{1 - \frac{2 \tilde{M}}{a}}, \nonumber \\
\mathfrak{\Gamma}^\pm &:= \iota(\Gamma^\pm ~ (\Theta \circ (\pm\eta)))
\ea
and we write for the Christoffel symbols (omitting indices) $\Gamma = \Gamma^+ ~ (\Theta \circ \eta) + \Gamma^- ~ (\Theta \circ (-\eta))$, so $\Gamma^\pm$ denote the Christoffel symbols on the positive and negative level sets of $\eta$, respectively. More explicitly the embedded Christoffel symbols and metric components look like (recall Eq. \eqref{eq:rinverse})
\ba
\left(\mathfrak{\Gamma}^\pm\right)^{\tilde{\eta}}_{~t_+t_+}
&= \pm \int_a^\infty \frac{\tilde{M}}{r^2} ~ \rho_\varepsilon(\tilde{\eta} - \tilde{\eta}_\pm(r)) ~ \mathrm{d}r, \quad
\left(\mathfrak{\Gamma}^\pm\right)^{\tilde{\eta}}_{~\vartheta\vartheta}
= \mp \int_a^\infty r ~ \rho_\varepsilon(\tilde{\eta} - \tilde{\eta}_\pm(r)) ~ \mathrm{d}r, \\
\left(\Gamma^\pm\right)^{\tilde{\eta}}_{~t_+t_+}(\tilde{\eta})
&= \pm \sqrt{1- \frac{2\tilde{M}}{r_\pm(\tilde{\eta})}} ~ \frac{\tilde{M}}{r_\pm(\tilde{\eta})^2},
\quad
\left(\Gamma^\pm\right)^{\tilde{\eta}}_{~\vartheta\vartheta}(\tilde{\eta})
= \mp r_\pm(\tilde{\eta}) ~ \sqrt{1 - \frac{2 \tilde{M}}{r_\pm(\tilde{\eta})}}, \\
\iota(g_{t_+t_+})
&=
- \int_a^\infty \sqrt{1-\frac{2\tilde{M}}{r}} ~ \left(\rho_\varepsilon(\tilde{\eta}-\tilde{\eta}_+(r)) + \rho_\varepsilon(\tilde{\eta}-\tilde{\eta}_-(r)) \right) ~ \mathrm{d}r, \\
\iota(g_{\vartheta\vartheta})
&=
\int_a^\infty r^2 ~ \sqrt{\frac{r}{r - 2\tilde{M}}} ~ \left( \rho_\varepsilon(\tilde{\eta} - \tilde{\eta}_+(r)) + \rho_\varepsilon(\tilde{\eta}-\tilde{\eta}_-(r)) \right) ~ \mathrm{d}r.
\ea
When one puts $\alpha_2 = 0$ then one gets exactly the classical results\footnote{Up to factors of the components of the metric since we did not raise any index in the stress-energy tensor like in \cite{Visser}.} as presented in \cite{Visser}, Chapter 15, which also describe the surface energy density $\sigma$ and the difference between the surface energy density and the tension $\nu_\vartheta$ along $\vartheta$ (everything multiplied with $\kappa$). Therefore we write for these expressions also $\sigma$ and $\sigma - \nu_\vartheta$. Observe that $\sigma$ is a negative energy density in the classical setting ($\alpha_2 = 0$) which is the reason why one speaks about that wormholes consists of exotic matter, \textit{exotic} because such matter was of course never observed yet. One also has $\kappa T^{\mathrm{NEC}}_{t_+t_+} \iota(g_{\varphi\varphi}) - \kappa T^{\mathrm{NEC}}_{\varphi\varphi} \iota(g_{t_+t_+}) = \sin^2(\vartheta) \left( \kappa T^{\mathrm{NEC}}_{t_+t_+} \iota(g_{\vartheta\vartheta}) - \kappa T^{\mathrm{NEC}}_{\vartheta\vartheta} \iota(g_{t_+t_+})\right)$ such that the left hand side is non-negative if and only if the right hand side is non-negative. Thence, one only has to check two inequalities instead of all three inequalities given in Thm. \ref{thm:NEC}.

We will now finally study the NEC but we will not do this explicitly since the terms are rather complicated (as one can see) such that a rigorous study would exceed this paper. Therefore we will state a conjecture about the NEC and we will motivate this by minimising again the microstructure.

\begin{conjecture}[Satisfied NEC]
\leavevmode\newline
When $\alpha_2 > 0$ then there is a constant $b > 2 \tilde{M}$ such that the NEC holds for all $a \in \left(2\tilde{M},b\right]$.
\end{conjecture}

\begin{motivation}
\leavevmode\newline
To motivate the conjecture let us simplify the terms like one may do this in physics. As we already discussed one can see that the supports of all the terms for $\varepsilon \to 0^+$ is the thin shell, \textit{i.e.} integrating the arising terms against a bump function with support off the thin shell one gets 0 for $\varepsilon \to 0^+$. Thus, one would think that one basically only needs the values on the thin shell. Hence the idea is to look at both sides/universes off the thin shell (where then everything is smooth and classical) and then to approach the thin shell by replacing the smooth functions in front of the delta distributions somehow by their constant values on the thin shell, \textit{e.g.} replacing the embedded metric by its value on the thin shell viewed of both sides and weighting the results of both sides off the thin shell with the corresponding Heaviside distribution (the same for the Christoffel symbols and for the products of both). By this procedure one avoids the derivatives of the Heaviside distribution since one does initially not directly look at the thin shell and one only has to smooth the Heaviside distribution. This is of course not rigorous, it is more like reducing again the microstructure as an ansatz to investigate the problem of non-negativity since one reduces the sequence structure to the Heaviside distributions. But beware that the derivatives of the delta distribution are of higher order, \textit{i.e.} they need also the information about the derivatives of the bump functions which can be seen by the following distributional properties for smooth functions $f$ (which can be derived very easily in the distributional framework and clearly extends to the Colombeau algebra on the level of association)
\bas
f \delta'
&=
-f'(0) ~ \delta + f(0) ~ \delta', \\
f \delta''
&=
f''(0) ~ \delta - 2 f'(0) ~ \delta' + f(0) ~ \delta''.
\eas
So one sees that for derivatives of the delta distribution it is not allowed in general to replace the smooth function in front by its value at zero. Thence one could ask if something similar happens at the square of the delta  distribution. This can not be answered easily and for simplicity we replace the function only with its value on the thin shell but this have to be discussed of course (therefore again: this is only a motivation of a conjecture and not a proof). One could try to prove this by showing that $f\delta = f(0) ~ \delta$ is stable in sense of association under a multiplication with an additional $\delta$.
\newline So one has overall the following ansatz of replacing the terms: For $f$ smooth on both sides off the thin shell (but might be discontinuous on the thin shell) we write $f = f^+ (\Theta \circ \eta) + f^- (\Theta \circ (-\eta))$ and apply the ansatz
\ba
\iota(f) ~ \iota(\delta\circ \eta) \leadsto&~
f^+(0) ~ \iota(\Theta \circ \eta) ~ \iota(\delta \circ \eta) + f^-(0) ~ \iota(\Theta \circ (-\eta)) ~ \iota(\delta\circ \eta), \\
\iota(f) ~ \iota(\delta'\circ\eta) \leadsto&~
f^+(0) ~ \iota(\Theta \circ \eta) ~ \iota(\delta' \circ \eta) + f^-(0) ~ \iota(\Theta \circ (-\eta)) ~ \iota(\delta' \circ \eta) \nonumber \\
&-\left(f^+\right)'(0) ~ \iota(\Theta \circ \eta) ~ \iota(\delta \circ \eta) -\left(f^-\right)'(0) ~ \iota(\Theta \circ (-\eta)) ~ \iota(\delta \circ \eta), \\
\iota(f) ~ \iota(\delta'' \circ \eta) \leadsto&~
f^+(0) ~ \iota(\Theta \circ \eta) ~ \iota(\delta''\circ\eta) + f^-(0) ~ \iota(\Theta \circ (-\eta)) ~ \iota(\delta''\circ\eta) \nonumber \\
&-2\left(f^+\right)'(0) ~ \iota(\Theta \circ \eta) ~ \iota(\delta'\circ\eta) -2\left(f^-\right)'(0) ~ \iota(\Theta \circ (-\eta)) ~ \iota(\delta'\circ\eta) \nonumber \\
&+\left(f^+\right)''(0) ~\iota(\Theta \circ \eta) ~ \iota(\delta \circ \eta)
+ \left(f^-\right)''(0) ~\iota(\Theta \circ (-\eta)) ~ \iota(\delta \circ \eta)
\ea
and if $f$ is continuous at the thin shell
\ba
\iota(f) ~ \iota(\delta \circ \eta)^2 \leadsto&~
f(0) ~ \iota(\delta\circ\eta)^2.
\ea
Using this idea the following term is replaced by (recall also Eq. \eqref{eq:jump} and also that the Christoffel symbols are discontinuous at the thin shell)
\bas
&\left( \left(\mathfrak{\Gamma}^+\right)^{\tilde{\eta}}_{~t_+t_+} + \left(\mathfrak{\Gamma}^-\right)^{\tilde{\eta}}_{~t_+t_+} \right) ~ \iota(\delta' \circ \eta) \\
&\leadsto
\Biggl(
\underbrace{\sqrt{1 - \frac{2 \tilde{M}}{a}} ~ \frac{\tilde{M}}{a^2}}_{\leadsto "f^+(0)"} ~ \underbrace{\iota(\Theta \circ \eta)}_{\text{since this is the part of } \mathfrak{\Gamma}^+}
- \sqrt{1 - \frac{2 \tilde{M}}{a}} ~ \frac{\tilde{M}}{a^2} ~ \iota(\Theta \circ (-\eta))
\Biggr)
~ \iota(\delta' \circ \eta) \\
&\quad~
- \Biggl(
\underbrace{\sqrt{1-\frac{2 \tilde{M}}{a}}}_{= \frac{\mathrm{d}r_+}{\mathrm{d}\tilde{\eta}}(0)} \underbrace{\left(\frac{1}{\sqrt{1-\frac{2 \tilde{M}}{a}}} \frac{\tilde{M}^2}{a^4} - 2 ~ \sqrt{1-\frac{2 \tilde{M}}{a}} ~ \frac{\tilde{M}}{a^3} \right)}_{\leadsto "\frac{\mathrm{d}f^+}{\mathrm{d}r_+}(0)"}
~ \iota(\Theta \circ \eta) \\
&\qquad\quad
+ \underbrace{\sqrt{1-\frac{2 \tilde{M}}{a}}}_{= - \frac{\mathrm{d}r_-}{\mathrm{d}\tilde{\eta}}(0)} \underbrace{\left(\frac{1}{\sqrt{1-\frac{2 \tilde{M}}{a}}} \frac{\tilde{M}^2}{a^4} - 2 ~ \sqrt{1-\frac{2 \tilde{M}}{a}} ~ \frac{\tilde{M}}{a^3} \right)}_{\leadsto "-\frac{\mathrm{d}f^-}{\mathrm{d}r_-}(0)"}
~ \iota(\Theta \circ (-\eta))
\Biggr)
~ \iota(\delta \circ \eta)
\eas
To simplify this we now calculate the multiplications of the Heaviside distribution with the derivatives of the delta distribution which can be calculated on the level of association. As mentioned in Rem. \ref{rem:association} we know
\bas
\iota(\Theta \circ (\pm \eta)) ~ \iota(\delta \circ \eta)
&\approx
\frac{1}{2} ~ \iota(\delta \circ \eta),
\eas
differentiating both equations along the vector field $X := \frac{\mathrm{grad}(\eta)}{||\mathrm{grad}(\eta)||^2} = \partial_{\tilde{\eta}}$ one gets (using Def. \ref{def:compwitheta} and Prop. \ref{prop:deltcompetaexp})
\bas
&&
\pm (\iota(\delta \circ \eta))^2
+ \iota(\Theta \circ (\pm \eta)) ~ \iota(\delta' \circ \eta)
&\approx
\frac{1}{2} ~ \iota(\delta' \circ \eta) \\
&\Leftrightarrow&
\iota(\Theta \circ (\pm \eta)) ~ \iota(\delta' \circ \eta)
&\approx
\frac{1}{2} ~ \iota(\delta' \circ \eta)
\mp (\iota(\delta \circ \eta))^2.
\eas
Differentiating this again along $X$ one gets similarly
\bas
&&
\pm \iota(\delta \circ \eta) ~ \iota(\delta' \circ \eta)
+ \iota(\Theta \circ (\pm \eta)) ~ \iota(\delta'' \circ \eta)
&\approx
\frac{1}{2} ~ \iota(\delta'' \circ \eta)
\mp 2 ~ \iota(\delta \circ \eta) ~ \iota(\delta' \circ \eta) \\
&\Leftrightarrow&
\iota(\Theta \circ (\pm \eta)) ~ \iota(\delta'' \circ \eta)
&\approx
\frac{1}{2} ~ \iota(\delta'' \circ \eta)
\mp 3 ~ \iota(\delta \circ \eta) ~ \iota(\delta' \circ \eta),
\eas
so one gets
\bas
\iota(\Theta \circ (\eta)) ~ \iota(\delta'' \circ \eta)
+ \iota(\Theta \circ (- \eta)) ~ \iota(\delta'' \circ \eta)
&\approx
\iota(\delta'' \circ \eta), \\
\iota(\Theta \circ (\eta)) ~ \iota(\delta' \circ \eta)
+ \iota(\Theta \circ (- \eta)) ~ \iota(\delta' \circ \eta)
&\approx
\iota(\delta' \circ \eta), \\
\iota(\Theta \circ (\eta)) ~ \iota(\delta' \circ \eta)
- \iota(\Theta \circ (- \eta)) ~ \iota(\delta' \circ \eta)
&\approx
- 2 (\iota(\delta \circ \eta))^2.
\eas
These results we will use in the following. Also observe that the third equality shows that the first derivative of the delta distribution can lead to a term of constant sign which one would normally not expect for the first derivative of the delta distribution. We will be able to use these results basically because of the equality of both masses on both sides of the thin shell, such that one might think that these results are physically unstable. One also sees by the first two equations that our approach makes sense with respect to the derivatives of the delta distribution (one also has clearly something similar for the delta distribution itself), so it not automatically fails, because in the smooth case (so really everything is smooth) the Heaviside distributions should not influence anything and on the level of association this would hold then (using the first two equations) such that one gets the expected equations for the multiplication $f \delta$ \textit{etc.}.
\newline Inserting these in the term of the Christoffel symbols above one gets up to association
\bas
&\left( \left(\mathfrak{\Gamma}^+\right)^{\tilde{\eta}}_{~t_+t_+} + \left(\mathfrak{\Gamma}^-\right)^{\tilde{\eta}}_{~t_+t_+} \right) ~ \iota(\delta' \circ \eta) \\
&\leadsto
\dots \approx
- 2 \sqrt{1- \frac{2 \tilde{M}}{a}} ~ \frac{\tilde{M}}{a^2} ~ (\iota(\delta \circ \eta))^2
+ \left( 2 \left( 1 - \frac{2 \tilde{M}}{a} \right) ~ \frac{\tilde{M}}{a^3} - \frac{\tilde{M}^2}{a^4} \right) ~ \iota(\delta \circ \eta).
\eas
Similarly one calculates the other terms
\bas
&-\iota^0_2(g_{t_+t_+}) ~ \iota(\delta'' \circ \eta)
\leadsto
\dots \\
&\approx
\left( 1 - \frac{2 \tilde{M}}{a} \right) \iota(\delta'' \circ \eta)
+ \frac{8 \tilde{M}}{a^2} \sqrt{1-\frac{2 \tilde{M}}{a}} ~ (\iota(\delta \circ \eta))^2
+ \left( \frac{2 \tilde{M}^2}{a^4} - \frac{4 \tilde{M}}{a^3} \left(1-\frac{2 \tilde{M}}{a}\right) \right) \iota(\delta \circ \eta), \\
&- \iota^0_2(g_{t_+t_+}) \left( \left(\mathfrak{\Gamma}^-\right)^{\tilde{\eta}}_{~\vartheta\vartheta} + \left(\mathfrak{\Gamma}^+\right)^{\tilde{\eta}}_{~\vartheta\vartheta} \right) ~ \iota(\delta' \circ \eta) \\
&\leadsto
\dots 
\approx
2 a \left( 1 - \frac{2 \tilde{M}}{a} \right)^{3/2} ~ (\iota(\delta \circ \eta))^2
+ \left(1 - \frac{2 \tilde{M}}{a}\right) \left( 1 + \frac{\tilde{M}}{a} \right) ~ \iota(\delta \circ \eta),\\
&\iota^0_2(g_{\vartheta\vartheta}) \left( \left(\mathfrak{\Gamma}^+\right)^{\tilde{\eta}}_{~t_+t_+} + \left(\mathfrak{\Gamma}^-\right)^{\tilde{\eta}}_{~t_+t_+} \right) ~ \iota(\delta' \circ \eta) \\
&\leadsto
\dots
\approx
- 2 \tilde{M} \sqrt{1 - \frac{2 \tilde{M}}{a}} ~ (\iota(\delta \circ \eta))^2
- \frac{\tilde{M}^2}{a^2} ~ \iota(\delta \circ \eta).
\eas
Inserting these results into $\kappa \sigma$ and $\kappa (\sigma - \nu_\vartheta)$ (Eq. \eqref{ineq1} and \eqref{ineq2}) one gets (after some straightforward calculations)
\bas
\kappa \sigma
&\leadsto \dots \approx
\left(
- \frac{4}{a} ~ \left(1-\frac{2 \tilde{M}}{a}\right)^{3/2}
+ 2 \alpha_2 \alpha \left( \frac{3 \tilde{M}^2}{a^4} - \frac{6 \tilde{M}}{a^3} \left( 1 - \frac{2 \tilde{M}}{a} \right) \right)
\right) ~ \iota(\delta \circ \eta) \\
&\quad~
+ 2 \alpha_2 \alpha \left( 1 - \frac{2 \tilde{M}}{a} \right) \iota(\delta'' \circ \eta) \\
&\quad~
+ \left(
\left( 1 - \frac{2 \tilde{M}}{a} \right)  \alpha^2
- 2 \alpha \beta
+ 2 \alpha ~ \frac{10 \tilde{M}}{a^2} \sqrt{1-\frac{2 \tilde{M}}{a}}
\right) \alpha_2 ~ (\iota(\delta \circ \eta))^2, \\
\kappa (\sigma - \nu_\vartheta)
&\leadsto \dots \\
&\approx \left( \left( 6 \tilde{M} -2 a \right) \sqrt{1 - \frac{2 \tilde{M}}{a}}
+ 2 \alpha_2 \alpha \left( \frac{\tilde{M}^2}{a^2} - \left(1 - \frac{2 \tilde{M}}{a}\right) \left( 1 + \frac{\tilde{M}}{a} \right) \right) \right) ~ \iota(\delta \circ \eta) \\
&
+ 2 \alpha_2 \alpha
\left(
\gamma \left( 1 - \frac{2 \tilde{M}}{a} \right)
- \beta a^2
- 2 a \left( 1 - \frac{2 \tilde{M}}{a} \right)^{3/2}
+ 2 \tilde{M} \sqrt{1 - \frac{2 \tilde{M}}{a}}
\right)
~ (\iota(\delta \circ \eta))^2
\eas
Let us now discuss these terms. We have
\bas
&&
0 &>
- \frac{4}{a} ~ \left(1-\frac{2 \tilde{M}}{a}\right)^{3/2}
=
\beta - \frac{\alpha}{2} \left( 1 - \frac{2 \tilde{M}}{a} \right) \\
&\stackrel{\cdot (- 2 \alpha)}{\Leftrightarrow}&
0 &<
\left( 1 - \frac{2 \tilde{M}}{a} \right) \alpha^2 - 2 \alpha \beta
\eas
and therefore it might be natural to put $\alpha_2 > 0$ such that the term proportional to the square of the delta distribution in $\kappa \sigma$ is clearly non-negative ($\alpha_2 = 0$ would clearly fail since then one has the classical case). Thus, let us put $\alpha_2 > 0$. Inserting the definitions of the constants $\alpha, \beta, \gamma$ into $\kappa (\sigma - \nu_\vartheta)$ one easily sees that the term proportional to the square of the delta distribution is zero, thus one only has a direct change of the classical term, \textit{i.e.} (again after some straightforward calculation)
\bas
&\kappa (\sigma - \nu_\vartheta) \leadsto \dots \approx
\left( 6 \tilde{M} - 2 a \right) \sqrt{1 - \frac{2 \tilde{M}}{a}} ~ \iota(\delta \circ \eta) \\
&
+ 2 \underbrace{\alpha_2 \left( \sqrt{\frac{a}{a-2\tilde{M}}} ~ \frac{4 \tilde{M}}{a^2} + \frac{8}{a} ~ \sqrt{1-\frac{2 \tilde{M}}{a}} \right)}_{> 0} \left( \frac{\tilde{M}^2}{a^2} - \left(1 - \frac{2 \tilde{M}}{a}\right) \left( 1 + \frac{\tilde{M}}{a} \right) \right) ~ \iota(\delta \circ \eta),
\eas
the first summand is clearly non-negative when one has $a \in \left(2 \tilde{M}, 3 \tilde{M}\right]$ ($a > 2 \tilde{M}$ to prevent an event horizon as usual; $a = 2 \tilde{M}$ would not disturb the non-negativity of course). This also shows that in the pure classical case ($\alpha_2=0$) this inequality is satisfied for small $a$. The problem in the classical case was especially $\sigma$ at which we look later and which is the reason that the NEC does not hold in the classical case. Let us now look when the second summand of $\kappa (\sigma - \nu_\vartheta)$ is non-negative using $\alpha_2 > 0$, \textit{i.e.} (using again that $a > 2 \tilde{M} > 0$ holds)
\bas
&&
0
&\stackrel{!}{\leq}
\frac{\tilde{M}^2}{a^2}
- \left(1 - \frac{2 \tilde{M}}{a}\right) \left( 1 + \frac{\tilde{M}}{a} \right)
=
\frac{\tilde{M}^2}{a^2}
- 1
- \frac{\tilde{M}}{a}
+ \frac{2 \tilde{M}}{a}
+ \frac{2 \tilde{M}^2}{a^2}
=
\frac{3 \tilde{M}^2}{a^2}
+ \frac{\tilde{M}}{a}
- 1
\eas
\bas
&\Leftrightarrow&
0
&\stackrel{!}{\leq}
3 \tilde{M}^2 + \tilde{M} a - a^2
&&\Leftrightarrow&
a
&\in
\left( 2 \tilde{M}, \frac{\tilde{M}}{2} \left( 1 + \sqrt{13} \right) \right].
\eas
Since $a_1 := \frac{\tilde{M}}{2} \left( 1 + \sqrt{13} \right) > 2 \tilde{M}$ one can clearly find an $a$ such that also the second summand of $\kappa (\sigma - \nu_\vartheta)$ is non-negative. Because also $a_1 < 3 \tilde{M}$ one can say that the following statement holds (if our ansatz is correct)
\bas
a \in \left(2 \tilde{M}, a_1 \right]
\Rightarrow
\kappa (\sigma - \nu_\vartheta) \geq 0,
\eas
which is clearly independent of $\alpha_2$. Observe that the range for $a$ is here smaller than in the classical case where one only needs $a \in (2 \tilde{M}, 3 \tilde{M}]$. One can also see that for $a \longrightarrow \left( 2 \tilde{M} \right)^+$ the second summand diverges to positive infinity, \textit{i.e.} it can be made arbitrarily big. This might be only singularity given by the coordinate singularity at the event horizon, \textit{i.e.} this singularity might be not physical.
\newline Now let us look at $\kappa \sigma$; inserting $\alpha, \beta$ and $\gamma$ one gets after some straightforward calculation
\bas
&\kappa \sigma \\
&\leadsto \dots \approx 
\frac{4}{a} \left( \frac{6 \alpha_2 \tilde{M}}{a^3} \left( \frac{\tilde{M}^2}{a^2} \sqrt{\frac{a}{a-2 \tilde{M}}}
- 4 \left( 1 - \frac{2 \tilde{M}}{a} \right)^{3/2} \right)
- \left(1-\frac{2 \tilde{M}}{a}\right)^{3/2} \right) ~ \iota(\delta \circ \eta) \\
&\quad
+ \frac{8 \alpha_2}{a} \sqrt{1 - \frac{2 \tilde{M}}{a}} \left( 2 - \frac{3 \tilde{M}}{a} \right) ~ \iota(\delta'' \circ \eta) \\
&\quad
+ \frac{16 \alpha_2}{a^2} \left( \frac{\tilde{M}}{a} + 2 \left( 1 - \frac{2 \tilde{M}}{a} \right) \right)
\left( \frac{5 \tilde{M}}{a} + 2 \left( 1 - \frac{2 \tilde{M}}{a} \right) \right) ~ (\iota(\delta \circ \eta))^2.
\eas
One can clearly see that the term proportional to the square of the delta distribution is non-negative for all $a > 2 \tilde{M}$ by $\alpha_2>0$. Also the term proportional to the delta distribution is clearly non-negative for $a \in \left(2 \tilde{M}, a_2\right]$ for some $a_2 > 2 \tilde{M}$ since the positive term $\frac{24 \alpha_2 \tilde{M}^3}{a^6} \sqrt{\frac{a}{a - 2 \tilde{M}}}$ diverges for $a \to \left( 2 \tilde{M} \right)^+$ while the negative terms are vanishing (this divergence might come again only by the coordinate singularity). So when the term proportional to the second derivative of the delta distribution would be not there then one would have clearly a non-negative term. For $a \to \left( 2 \tilde{M} \right)^+$ one sees that this problematic term vanishes while the term of the squared delta distribution goes to $\frac{80 \alpha_2 \tilde{M}^2}{a^4} ~ (\iota(\delta \circ \eta))^2 \geq 0$ and the term proportional to the delta distribution grows arbitrarily big to positive infinity as discussed before. So one could think, that making $a$ small enough might lead to a satisfied inequality and that is the reason that I think that the NEC can be satisfied since the second derivative of the delta distribution can be controlled by lowering $a$ (the other inequality with $\nu_\vartheta$ is clearly already satisfied for small $a$ as discussed). In the classical case the energy density vanishes for $a = 2 \tilde{M}$, so then the NEC is satisfied but in our case now it does not vanish, it is arbitrarily big, so one could argue that it is positive. By some sense of continuity one could mention that for $a - 2 \tilde{M}$ small one may achieve a non-negative term such that the NEC should be satisfied for an $a > 2 \tilde{M}$.
\newline But beware that in \textit{e.g.} $\mathcal{D}'\left(\mathds{R}\right)$ the sum $c \delta + b \delta''$ with $c \geq 0$ and $b \in \mathds{R}$ is non-negative if and only if $b=0$. This statement is clear for $c=0$ since $\delta''(f) = f''(0)$ for all smooth functions $f$ (\textit{i.e.} not a non-negative distribution; recall Def. \ref{def:nonnegdistr}). For $c>0$ it is clear that it is non-negative (actually positive) for $b=0$; a more general explanation for any $b$ and $c$ would be by assuming that $c \delta + b \delta''$ is non-negative, \textit{i.e.} it is a signed measure (see the previous discussion about generalizing the sense of non-negativity, especially Thm. \ref{thm:geilestheorem}). Since the delta distribution is also a measure one could conclude that $b \delta''$ is also a signed measure or equivalently a distribution of zero order by the representation theorem of  Riesz-Markov-Saks-Kakutani which is clearly a contradiction for $b\neq0$ since $\delta''$ is a distribution of second order.
\newline Thus, one could argue that the approach above fails since one may only achieve non-negativity for $a = 2 \tilde{M}$; so making $a- 2\tilde{M}$ small but unequal to zero may fail. Therefore, when the conjecture is correct, the role of the positive term of the square of the delta distribution is very crucial. When $\varepsilon \to 0^+$ the positive term proportional to the square of the delta distribution diverges such that one might use this to show non-negativity. But also $\int (\iota(\delta''\circ \eta))_- ~ w$ diverges for $\varepsilon \to 0^+$ since otherwise, if it would exist, one could split $\int \iota(\delta''\circ \eta) ~ w$ (which exists clearly for $\varepsilon \to 0^+$) into the integrals over the negative and positive part. Then also $\int (\iota(\delta''\circ \eta))_+ ~ w$ would exist for $\varepsilon \to 0^+$ but then by the argumentation of the proof of Thm. \ref{thm:geilestheorem} \textit{etc.} $\delta'' \circ \eta$ would be a signed measure which is a contradiction as discussed before. Therefore one has to argue if the square of the delta distribution can compensate the negative part of $\iota(\delta'' \circ \eta)$; maybe here small $a$ can then help.
\end{motivation}
 
\section{Conclusion}

In this paper we saw that one can avoid junction conditions and higher conditions on
the regularity of the metric by using the Colombeau algebra, a generalized framework of the distributional framework, allowing us to study wormholes in theories of gravity which are normally excluded in other papers due to the junction conditions. By this manner we were able to look at a thin shell wormhole
described by a continuous metric embedded into $F(R)$-gravity for arbitrary polynomials
$F$. We also discussed that one can not even rigorously take the divergence of the arising
stress-energy tensor in the classical distributional framework but also this is also solved
in the Colombeau algebra, moreover, one can expect that the known results carry over in
sense of association although they were not derived mathematical rigorous in the classical setting.
We were also able to find a slightly generalized definition of non-negativity such that
(w.r.t. a fixed suitable volume form) on one hand one has no increased set of nonnegative
distributions but on the other hand one has non-negative elements which are
not embedded distributions, e.g. even powers of delta distributions which one would
like to regard as non-negative. Using this together with a suitable definition of lightlike
vector fields one can then naturally generalize
the NEC. We then saw that we have a generalized setting for stress-energy tensors of
Hawking-Ellis type I by Thm. \ref{thm:NEC}. Using this and the approach of minimizing the so-called microstructure we have seen
for quadratic $F$ that there might be a suitable radius of the wormhole throat such that
the NEC is satisfied for our taken thin shell wormhole, especially due to additional terms
in the function of the energy density given by the new arising physical information of the Colombeau algebra due to the microstructure. In the classical situation one does not have any information about the microstructure and, thus, one speaks about exotic matter violating the NEC, especially the energy density is negative in the classical situation while this may not the case with the microstructure.

But we also mentioned the problem of the arising
of the second derivative of the delta distribution. To solve this problem it might be thus
important to understand the square of the delta distribution in sense of the Colombeau
algebra much better. As a first approach one could try to fix an suitable admissible
mollifier and to calculate the arising terms by using an explicit form of the mollifier.
But one also has to solve the problem of the microstructure. For a rigorous proof one is not allowed to use the approach of the conjecture above. Therefore one has to work with the complete framework of the Colombeau algebra, including the new information given by the microstructure. To work with the microstructure one has to characterize it, especially one may have to work out a more detailed picture about its physical interpretation. This would mean that one has to extend the use of the Colombeau algebra in the fields of physics like mentioned in \cite{VickerAnfang} where it is shortly explained why the typical distributional approach has to fail for matter distributions supported in manifolds with codimension of at least 2. \cite{VickerAnfang} starts with a general description of distributions in the standard theory of general relativity and concludes when one can use the classical framework of distributions by stating Theorem 2.2 and the passages before and after it. Especially Theorem 2.2 states that distributions can only be used for thin shells of matter but not for (cosmic) strings and point particles. Therefore \cite{VickerAnfang} introduces also the Colombeau algebra and in Section 6 another example with microstructure is presented, \textit{the thin string limit of cosmic strings}. This example is about "the thin string limit of the field equations for an infinite length gravitating straight cosmic string described by a complex scalar field coupled to a U(1) gauge field" (\cite{VickerAnfang}). There it is mentioned that in Physics such line sources have to satisfy the condition 
\ba
\Delta \Phi = 8 \pi \mu, 
\label{eq:cosmic}
\ea
where $\Delta \Phi$ is the angular deficit and $\mu$ describes the mass per unit length. It is well known that this type of cosmic strings satisfy this condition if and only if it is critically coupled, \textit{i.e.} its vector and scalar fields have equal masses. Normally it is believed that this type of cosmic strings is not well-defined since Condition \eqref{eq:cosmic} depends upon the matter of the string and, thus, can be too easily violated.

However, in \cite{VickerAnfang} it is discussed that Condition \eqref{eq:cosmic} is only needed when one wants to stay in the distributional framework, similar to the statement of the junction conditions which we avoided here in this paper; similarly one needs the Colombeau algebra when one drops Condition \eqref{eq:cosmic}. The main problem of this type of cosmic string is that its stress-energy tensor consists not only of a classical term proportional to a delta distribution with support on the string, it also contains terms like $d ~ \frac{2xy}{x^2+y^2} ~ \delta^{(2)}(x,y)$ where $x$ and $y$ are the Euclidean coordinates for the radial position to the string (\textit{i.e.} the string is the $z$-axis); $d$ is some constant which vanishes for critically coupled strings. Such terms are not well-defined at the string since the limit is directional dependent, \textit{i.e.} such terms do not describe a distribution but for $d=0$ the stress-energy tensor is a typical distribution. One attempt to define it for $d \neq 0$ is the following (and remember the methods presented here in this paper), still following \cite{VickerAnfang} and the references therein: 
\newline
\newline The matter of such strings should be supported at the string, but, as one saw, the internal structure of such strings seems to be too complicated as if it could be described by a simple delta distribution such that one has to regard the internal structure of the matter.\footnote{Remember: Due to the nature of the delta distribution one looses any information off the support of the delta distribution by $f(x) ~ \delta(x) = f(0) ~ \delta(x)$, \textit{i.e.} structure of the matter along $x$ is lost and trivial while one can keep such information in the Colombeau algebra in form of microstructure.} To describe the matter and its internal structure one has the approach to make some physical parameters like the coupling constant/charge to be dependent on some parameter $\varepsilon$ such that the original value for the thin string limit is given by taking $\varepsilon$ to zero. Like for the thin shell wormhole one may think about this ansatz as approximating a string with a cylinder which allows a more complicated internal structure of the matter along the radial direction. These sequences will be elements of the Colombeau algebra and one will have additional sequential information like for our thin shell wormhole; and again, due to the singular terms in the stress-energy tensor some sequential information of the cylindrical internal matter distribution survives\footnote{The limit does not even exist, otherwise a distributional description would be possible.} (which normally vanishes in the distributional limit) because of multiplications with such singular terms. Only for $d = 0$ this microstructure still vanishes and will not play any role for $\varepsilon \to 0$. Hence, also here one has some form of microstructure which may be interpreted physically, in \cite{VickerAnfang} this is also called \textit{distributional microstructure}.

Thence, such conditions like the junction conditions and Condition \eqref{eq:cosmic} may rather be only a technical condition to avoid new mathematical frameworks and technical problems in the mathematics than be something physical. When one looks at their derivation then it is only often about something "not well-defined" and these conditions describe how to get rid off these terms. But I hope that the examples mentioned here may show that one could find new physical descriptions and data by allowing new mathematical frameworks where these critical terms are well-defined. Similar to the introduction of the delta distribution to describe the Maxwell equations for point particles in flat spacetime.

Another useful examples where one can use the Colombeau algebra is \textit{e.g.} given in \cite{Schwarzschild} (a description of the Schwarzschild spacetime in sense of Colombeau which allows to keep the origin of the space in the coordinates without worrying about the singularity), \cite{GenRiem} (at the end, page 21, a discussion about the geodesic equation of impulsive pp-waves) and \cite{Multiplikation} (for applications in classical Electrodynamics). In the beginning of \cite{VickerVorarbeit} is also another concise description about the failure of the distributional framework in special (but not unimportant) situations and the use of the Colombeau algebra in such a situation; especially section 5 (and its references) about ultrarelativistic Black Holes might be interesting because this is normally regarded as an example of the "failure regions" of general relativity. However, there it is discussed that general relativity would not fail anymore in such cases when one allows the Colombeau algebra and microstructure in general relativity. This is done by showing that the stress-energy tensor is more or less the square root of the delta distribution which can be defined in the Colombeau algebra, \textit{i.e.} the stress-energy tensor is non-zero which would resolve this situation. But there this stress-energy tensor and its microstructure would converge to a distribution, the zero function which is the reason why one believes that the ultrarelativistic Black Hole is a failure of general relativity, while it is still non-zero in sense of Colombeau. One could say it is an infinitesimal stress-energy tensor. So the difference to our case is that the microstructure and its stress-energy tensor has in that case a limit for $\varepsilon \to 0$ since there are not too strong singularities (and so one has not such conditions like the junction conditions in Physics for these situations). Concluding this paper, one might hence have to think about different types of microstructure.


\newpage



\renewcommand\refname{List of References}

\bibliography{Literatur}
\bibliographystyle{unsrt}


\end{document}